\documentclass[a4paper,12pt,reqno]{amsart}
\usepackage{amsmath,graphicx,graphics,amssymb}
\newtheorem{theo}{Theorem}

\newtheorem{cor}{Corollary}

\newtheorem{prop}{Proposition}

\numberwithin{equation}{section}

\newcommand{\M}{\operatorname{M}}

\newmuskip\pFqskip
\mathchardef\pFcomma=\mathcode`, 

\begin{document}

\title{Correlation of a macroscopic dent in a wedge with mixed boundary conditions}

\author{Mihai Ciucu}
\address{Department of Mathematics, Indiana University, Bloomington, Indiana 47405}

\thanks{Research supported in part by NSF grant DMS-1501052}

\begin{abstract} As part of our ongoing work on the enumeration of symmetry classes of lozenge tilings of hexagons with certain four-lobed structures removed from their center, we consider the case of the tilings which are both vertically and horizontally symmetric. In order to handle this, we need an extension of Kuo's graphical condensation method, which works in the presence of free boundary. Our results allow us to compute exactly the correlation in a sea of dimers of a macroscopic dent in a 90 degree wedge with mixed boundary conditions. We use previous results to compute the correlation of the corresponding symmetrized system with no boundary, and show that its fourth root has the same log-asymptotics as the correlation of the dent in the 90 degree wedge. This is the first result of this kind involving a macroscopic defect. It suggests that the connections between dimer systems with gaps and 2D electrostatics may be deeper that previously thought.
\end{abstract}

\maketitle

\section{Introduction}



In \cite{vf} we generalized MacMahon's \cite{MacM} beautiful formula stating that the number of lozenge tilings of a hexagonal region of side-lengths $x$, $y$, $z$, $x$, $y$, $z$ (in cyclic order) on a triangular lattice is
\begin{equation}
\prod_{i=1}^x\prod_{j=1}^y\prod_{k=1}^z \frac{i+j+k-1}{i+j+k-2},
\label{eaa}
\end{equation}
by showing that the number of lozenge tilings of hexagonal regions with a 4-lobed structure (called a shamrock) removed from their center is given by a product formula generalizing \eqref{eaa}. Motivated by the singularly elegant situation that all symmetry classes of lozenge tilings of a hexagon are given by equally beautiful formulas (see \cite{And}\cite{Sta}\cite{Kup}\cite{Ste}\cite{KKZ} and the survey \cite{Bres} for more recent developments), it is natural to consider the problem of enumerating the symmetry classes of tilings of these more general regions. Six new questions arise in this way. In \cite{symffa} we solved the cyclically symmetric and the cyclically symmetric and transpose complementary cases (invariance under rotation by 120 degrees, resp. invariance under the same plus reflection across vertical), and in \cite{symffb} we presented the transpose complementary case (invariance under reflection across vertical).

The purpose of this paper is to present the enumeration of a fourth case, that of symmetric and self complementary tilings (i.e., tilings that are horizontally symmetric and centrally symmetric). We achieve this by first generalizing the family of regions, and then proving product formulas for the number of tilings of the more general regions (see Theorem \ref{tbc} and Proposition~\ref{tca}).

Very useful for proving such formulas is Kuo's graphical condensation method \cite{KuoOne}\cite{KuoTwo}. However, the tilings we consider are tilings or regions for which part of their boundary is free (i.e. lozenges are allowed to protrude out halfway through those portions), and Kuo's original results did not deal with the case of a free boundary. We therefore need to first work out free boundary versions of Kuo's formulas.

We present three such free boundary analogs (see Theorem \ref{tba} and Corollary \ref{tbb}). The first one (which applies in a more general setting than the other two) is an eight-term recurrence. The other two are four-term recurrences (one being exactly Kuo's Pfaffian recurrence!) that can be deduced from the first one. One of the latter is what we use to prove our results.

Our results allow us to compute exactly the correlation in a sea of dimers of a macroscopic dent in a 90 degree wedge with mixed boundary conditions (see Theorem \ref{tda}). We use previous results to compute the correlation of the corresponding symmetrized system with no boundary(see Theorem \ref{tdb}), and show that its fourth root has the same log-asymptotics as the correlation of the dent in the 90 degree wedge (see Corollary \ref{tdd}). This is the first result of this kind involving a macroscopic defect. It suggests that the connections between dimer systems with gaps and 2D electrostatics may be deeper that previously thought.

\section{Statement of main results}

For a planar graph $G$ with weights on its edges and a distinguished subset $S$ of vertices on some face, we denote by $M_f(G)$ the sum of the weights\footnote{ The weight of a matching is the product of the weights of the edges in it.} of all the (not necessarily perfect) matchings of $G$ in which all the vertices that are not in $S$ are matched, but those in $S$ are free to be matched or not matched (the distinguished subset of vertices $S$ will be clear from context, so we do not need to include $S$ in the notation). Clearly, setting all the edge weights equal to 1 results in $M_f(G)$ simply counting all such matchings.

If $a,b,c,d\notin S$ are four vertices appearing in this cyclic order on the same face as the one containing the vertices in $S$, we say that $S$ is {\it $a,c$-separated} if there are no mutually disjoints paths $P_1,P_2,P_3$ in $G$ so that $P_1$ connects $a$ to $c$, $P_2$ connects $b$ to some vertex in $S$, and $P_3$ connects $d$ to some other vertex of $S$.

\begin{theo}
\label{tba}
Let $G$ be a weighted planar graph with the vertices $a$, $b$, $c$, $d$ appearing in that cyclic order on a face $F$ of $G$. Let $S$ be a subset of the vertices of $F$ that is disjoint with $\{a,b,c,d\}$, and assume that $S$ is $a,c$-separated.

Then we have
\begin{align}
&\!\!\!\!\!\!
\M_f(G)\M_f(G\setminus\{a,b,c,d\})
+
\M_f(G\setminus\{b,d\})\M_f(G\setminus\{a,c\})
\nonumber
\\
&\ \ \ \ \ \ \ \ \ \ \ \ \ \ \ \ \ \ \ \ \ \ \ \ \ \ \ \ \ \ \ \ \ 
+
\M_f(G\setminus\{b\})\M_f(G\setminus\{a,c,d\})
+
\M_f(G\setminus\{d\})\M_f(G\setminus\{a,b,c\})
\nonumber
\\
=
&
\M_f(G\setminus\{a,d\})\M_f(G\setminus\{b,c\})
+
\M_f(G\setminus\{a,b\})\M_f(G\setminus\{c,d\})
\nonumber
\\
&\ \ \ \ \ \ \ \ \ \ \ \ \ \ \ \ \ \ \ \ \ \ \ \ \ \ \ \ \ \ \ \ 
+
\M_f(G\setminus\{a\})\M_f(G\setminus\{b,c,d\})
+
\M_f(G\setminus\{a,b,d\})\M_f(G\setminus\{c\}).
\label{eba}
\end{align}
\end{theo}

\begin{proof} For any subgraph $H$ of $G$ containing the vertices in $S$, denote by $\mathcal M_f(H)$ the set of matchings of $H$ in which all vertices not in $S$ are matched, but the ones in $S$ are free to be matched or not matched. Patterned on the two sides of equation \eqref{eba}, consider the disjoint unions of Cartesian products
\begin{align}
&\!\!\!\!\!\!\!\!\!\!\!\!\!\!\!\!
\mathcal M_f(G)\times \mathcal M_f(G\setminus\{a,b,c,d\}) 
\cup
\mathcal M_f(G\setminus\{b,d\})\times\mathcal M_f(G\setminus\{a,c\})
\nonumber
\\
&\ \ \ \ \ \ 
\cup
\mathcal M_f(G\setminus\{b\})\times\mathcal M_f(G\setminus\{a,c,d\})
\cup
\mathcal M_f(G\setminus\{d\})\times\mathcal M_f(G\setminus\{a,b,c\})
\label{ebb}
\end{align}
and
\begin{align}
&\!\!\!\!\!\!\!\!\!\!\!\!\!\!\!\!
\mathcal M_f(G\setminus\{a,d\})\times \mathcal M_f(G\setminus\{b,c\}) 
\cup
\mathcal M_f(G\setminus\{a,b\})\times\mathcal M_f(G\setminus\{c,d\})
\nonumber
\\
&\ \ \ \ \ \ 
\cup
\mathcal M_f(G\setminus\{a\})\times\mathcal M_f(G\setminus\{b,c,d\})
\cup
\mathcal M_f(G\setminus\{a,b,d\})\times\mathcal M_f(G\setminus\{c\}).
\label{ebc}
\end{align}
For any element $(\mu,\nu)$ of \eqref{ebb} or \eqref{ebc}, think of the edges of $\mu$ as being marked by solid lines, and of the edges of $\nu$ as marked by dotted lines, {\it on the same copy of the graph $G$} (any edge common to $\mu$ and $\nu$ will be marked both solid and dotted, by two parallel arcs). Note that for all $(\mu,\nu)$ corresponding to cartesian products in \eqref{ebb}, $a$ is matched by a solid edge (this is the reason for the chosen order of factors in the terms of \eqref{eba}).

Define the weight of $(\mu,\nu)$ to be the product of the weight of $\mu$ and the weight of $\nu$. Then the total weight of the elements of the set \eqref{ebb} is equal to the left hand side of equation~\eqref{eba}, while the total weight of the elements of the set \eqref{ebc} equals the right hand side of \eqref{eba}. Therefore, to prove \eqref{eba} it suffices to construct a weight-preserving bijection between the sets~\eqref{ebb} and~\eqref{ebc}.

We construct such a bijection as follows. Let $(\mu,\nu)$ be an element of \eqref{ebb}. Map  $(\mu,\nu)$ to what we get from it by ``shifting along the path containing $a$.'' More precisely, note that when considering the edges of $\mu$ and $\nu$ together on the same copy of $G$, each of the vertices $a,b,c,d$ is incident to precisely one edge. All the other vertices of $G$ that are not in $S$ are incident to one solid edge and one dotted edge. Finally, each vertex in $S$ could be incident to no edge, to a single edge (solid or dotted), or to one solid and one dotted edge.

\begin{figure}[h]
\centerline{
\hfill
{\includegraphics[width=0.90\textwidth]{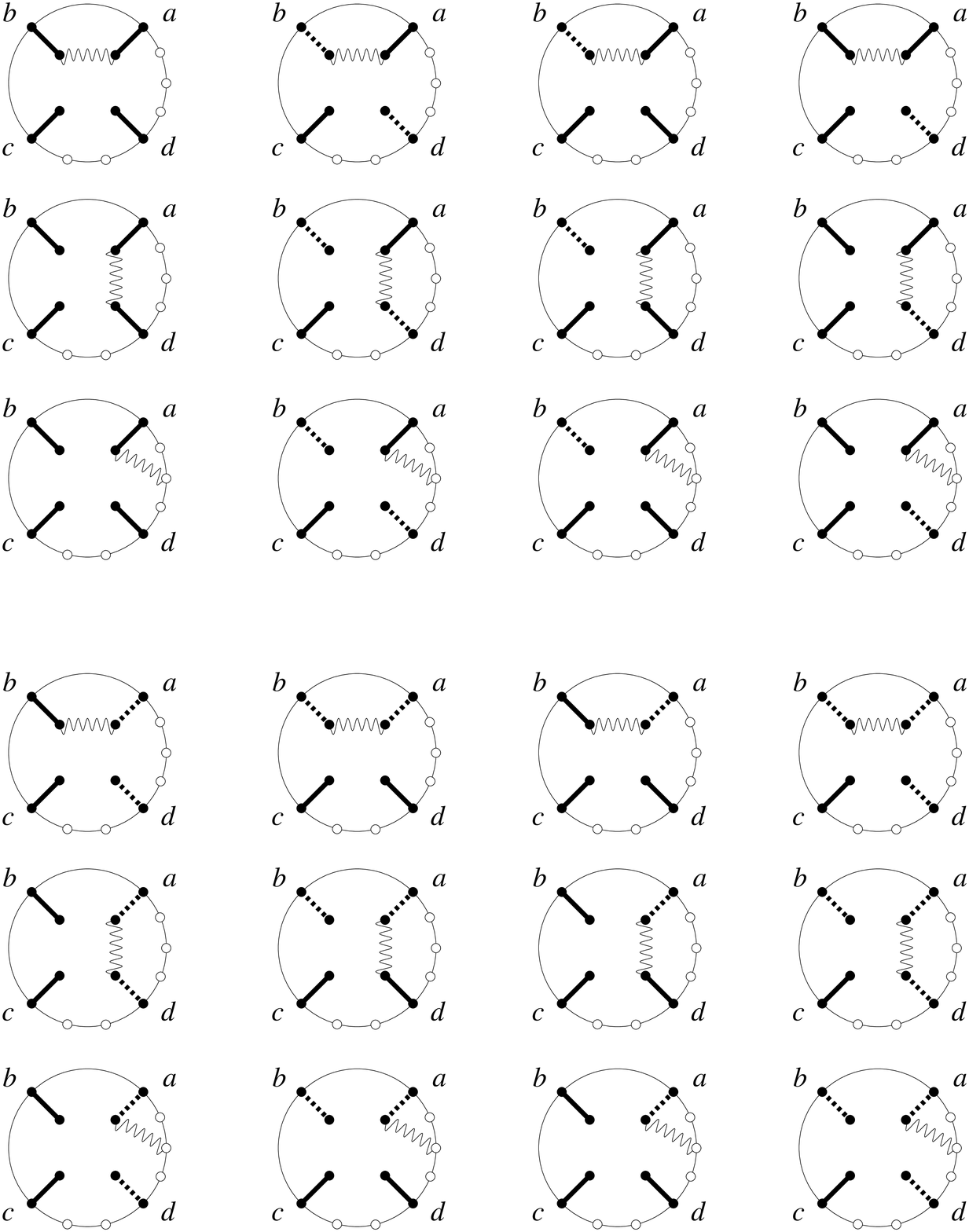}}
\hfill
}
\vskip-0.1in
\caption{Schematic representation of the bijection proving \eqref{eba}. Shifting along path containing $a$ matches the partition classes according to the pattern \newline
$A_1\ \ A_2\ \ A_3\ \ B_1\ \ B_2\ \ B_3\ \ C_1\ \ C_2\ \ C_3\ \ D_1\ \ D_2\ \ D_3$\newline
$B'_1\ \ A'_2\ \ C'_3\ \ A'_1\ \ B'_2\ \ D'_3\ \ C'_1\ \ D'_2\ \ B'_3\ \ D'_1\ \ C'_2\ \ A'_3$
}
\vskip-0.1in
\label{fba}
\end{figure}

This implies that $\mu\cup\nu$ is the disjoint union of (1) paths connecting each of $a$, $b$, $c$ and $d$ either to another element of $\{a,b,c,d\}$ or to some vertex in $S$; (2) paths (if any) connecting in pairs some of the vertices of $S$ not connected to $\{a,b,c,d\}$; and (3) cycles covering all the remaining vertices of~$G\setminus S$ and possibly some of the remaining vertices of $S$. Consider the path containing $a$, and change each solid edge in it to dotted, and each dotted edge to solid. Denote the resulting pair of matchings by $(\mu',\nu')$. 
Clearly, the weight of $(\mu',\nu')$ is the same as the weight of $(\mu,\nu)$. Therefore, it is enough to show that this map is a bijection.

To see this, we partition each of the four cartesian products in \eqref{ebb} into three classes, according to the three connection possibilities for vertex $a$: We gather those $(\mu,\nu)$ for which, in the superposition of $\mu$ and $\nu$, $a$ is connected by a path to $b$, into one class; those for which $a$ is connected by a path to $d$ into another class; and those for which $a$ is connected by a path to some vertex in $S$ into a third class (these partitions are represented schematically in the top half of Figure \ref{fba}).

The key to our proof (and the reason the above description forms a partition) is that, due to our assumption that $S$ is $a,c$-separated, the situation of $a$ being connected by a path to $c$ does not arise. 

Partition similarly each of the four cartesian products in \eqref{ebc} into three classes, according to the same three possible connection types for $a$. These are illustrated in the bottom half of Figure \ref{fba}. 

Under the above mapping $(\mu,\nu)\mapsto(\mu',\nu')$, each of the twelve classes of superpositions of matchings corresponding to the cartesian products in \eqref{ebb} turns out to be mapped bijectively to a different class of superpositions of matchings corresponding to the cartesian products in~\eqref{ebc}. The correspondence is indicated in Figure \ref{fba} (the top four groups of 3 ``balls'' are denoted from left to right by $A$, $B$, $C$, $D$, and the bottom ones by $A'$, $B'$, $C'$, $D'$; the subscript $i$ indicates that the $i$th ball from the group --- counting from the top --- is chosen).

Indeed, consider for instance from among the twelve classes of superpositions of matchings corresponding to the cartesian products in \eqref{ebb}, the class consisting of those $(\mu,\nu)$ corresponding to the first cartesian product in \eqref{ebb} for which $a$ is connected by a path to $b$ in the superposition of $\mu$ and $\nu$. Since both $a$ and $b$ are matched by solid edges, after we apply our construction to obtain $(\mu',\nu')$ (which recall consists in reversing the type of all edges in the path containing $a$ --- turning solid edges to dotted, and dotted to solid), both $a$ and $b$ will be matched by dotted edges. They will also clearly still be connected to one another. Therefore this class is mapped into the ``$a$ connected to $b$'' class of the second cartesian product in \eqref{ebc}. Since our map is clearly an involution, it establishes a bijection between these two classes.

Similarly, the ``$a$ connected to $d$'' class of the first cartesian product in \eqref{ebb} is mapped bijectively onto the ``$a$ connected to $d$'' class of the first cartesian product in \eqref{ebc}.

The remaining class of the first cartesian product in \eqref{ebb} consists of those $(\mu,\nu)$ for which, when superimposing $\mu$ and $\nu$, $a$ gets connected by a path to some vertex in $S$. As the edge matching $a$ in this path is solid, and we are reversing the types of the edges in this path to get $(\mu',\nu')$, the edge matching $a$ in the superposition of $(\mu',\nu')$ is dotted, and the edges matching $b$, $c$ and $d$ remain all solid. Thus this third class is mapped bijectively onto the ``$a$ connected to some vertex in $S$'' class of the third cartesian product in \eqref{ebc}.

All remaining bijective mappings of classes indicated in Figure \ref{fba} are justified similarly. This completes the proof. \end{proof}

\begin{cor}
\label{tbb}
Suppose we have all the hypotheses of Theorem \ref{tba}, and assume in addition that $S$ is also $b,d$-separated. Then we have
\medskip
\begin{align}
&\!\!\!\!\!\!
\M_f(G)\M_f(G\setminus\{a,b,c,d\})
+
\M_f(G\setminus\{b,d\})\M_f(G\setminus\{a,c\})
\nonumber
\\
&\ \ \ \ \ \ \ \ \ \ \ \ \ \ \ \ \ \ \ \ \ \ \ \ 
=
\M_f(G\setminus\{a,d\})\M_f(G\setminus\{b,c\})
+
\M_f(G\setminus\{a,b\})\M_f(G\setminus\{c,d\})
\label{ebd}
\end{align}
\medskip
and
\medskip
\begin{align}
&\!\!\!\!\!\!
\M_f(G\setminus\{b\})\M_f(G\setminus\{a,c,d\})
+
\M_f(G\setminus\{d\})\M_f(G\setminus\{a,b,c\})
\nonumber
\\
&\ \ \ \ \ \ \ \ \ \ \ \ \ \ \ \ \ \ \ \ \ \ \ \ 
=
\M_f(G\setminus\{a\})\M_f(G\setminus\{b,c,d\})
+
\M_f(G\setminus\{a,b,d\})\M_f(G\setminus\{c\}).
\label{ebe}
\end{align}
\end{cor}

\begin{proof} By our assumptions and Theorem \ref{tba}, equation \eqref{eba} holds. In addition, by applying Theorem \ref{tba} by viewing $S$ as being $b,d$-separated, we obtain
\medskip
\begin{align}
&\!\!\!\!\!\!
\M_f(G)\M_f(G\setminus\{a,b,c,d\})
+
\M_f(G\setminus\{a,c\})\M_f(G\setminus\{b,d\})
\nonumber
\\
&\ \ \ \ \ \ \ \ \ \ \ \ \ \ \ \ \ \ \ \ \ \ \ \ \ \ \ \ \ \ \ \ \ 
+
\M_f(G\setminus\{a\})\M_f(G\setminus\{b,c,d\})
+
\M_f(G\setminus\{c\})\M_f(G\setminus\{a,b,d\})
\nonumber
\\
=
&
\M_f(G\setminus\{a,b\})\M_f(G\setminus\{c,d\})
+
\M_f(G\setminus\{b,c\})\M_f(G\setminus\{a,d\})
\nonumber
\\
&\ \ \ \ \ \ \ \ \ \ \ \ \ \ \ \ \ \ \ \ \ \ \ \ \ \ \ \ \ \ \ \ 
+
\M_f(G\setminus\{b\})\M_f(G\setminus\{a,c,d\})
+
\M_f(G\setminus\{a,b,c\})\M_f(G\setminus\{d\}).
\label{ebf}
\end{align}
Equations \eqref{eba} and \eqref{ebf} imply \eqref{ebd} and \eqref{ebe}. \end{proof}

\bigskip
\parindent=0pt
{\bf Remark 1.} When $S=\emptyset$, equation \eqref{ebd} becomes Kuo's Proposition 1.1 of  \cite{KuoTwo}:
\medskip
\begin{align}
&\!\!\!\!\!\!
\M(G)\M_f(G\setminus\{a,b,c,d\})
+
\M(G\setminus\{b,d\})\M_f(G\setminus\{a,c\})
\nonumber
\\
&\ \ \ \ \ \ \ \ \ \ \ \ \ \ \ \ \ \ \ \ \ \ \ \ 
=
\M(G\setminus\{a,d\})\M_f(G\setminus\{b,c\})
+
\M(G\setminus\{a,b\})\M_f(G\setminus\{c,d\}),
\label{ebg}
\end{align}

\medskip
where $\M(G)$ stands for the number (or total weight, if $G$ is weighted) of {\it perfect} matchings of~$G$.

\parindent=15pt
If in addition $G$ is bipartite, has two more white than black vertices, and $a$, $b$, $c$, $d$ are all white, \eqref{ebd} becomes Kuo's Theorem 2.5 of \cite{KuoOne}:
\medskip

\begin{figure}[h]
\centerline{
\hfill
{\includegraphics[width=0.44\textwidth]{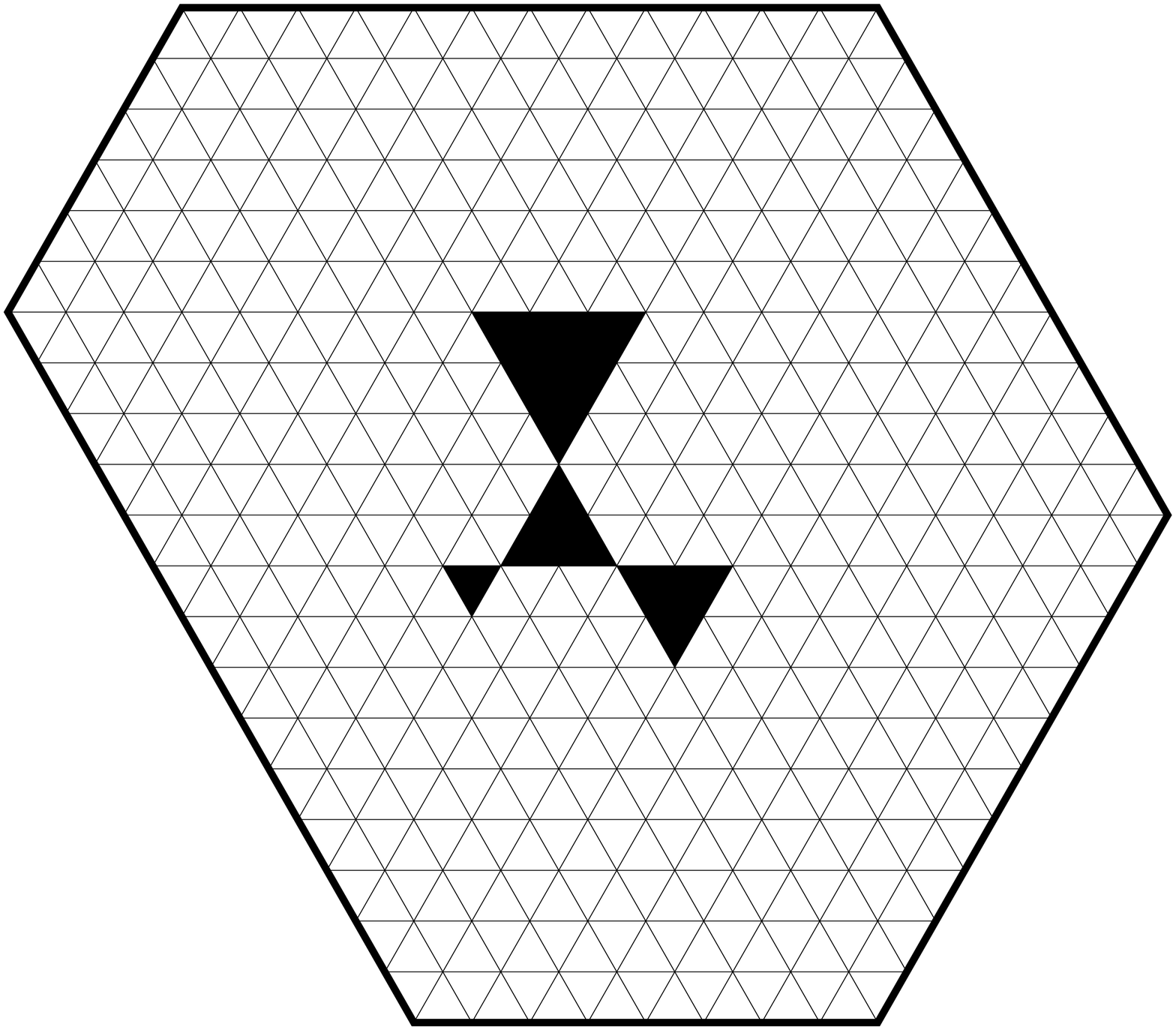}}
\hfill
{\includegraphics[width=0.40\textwidth]{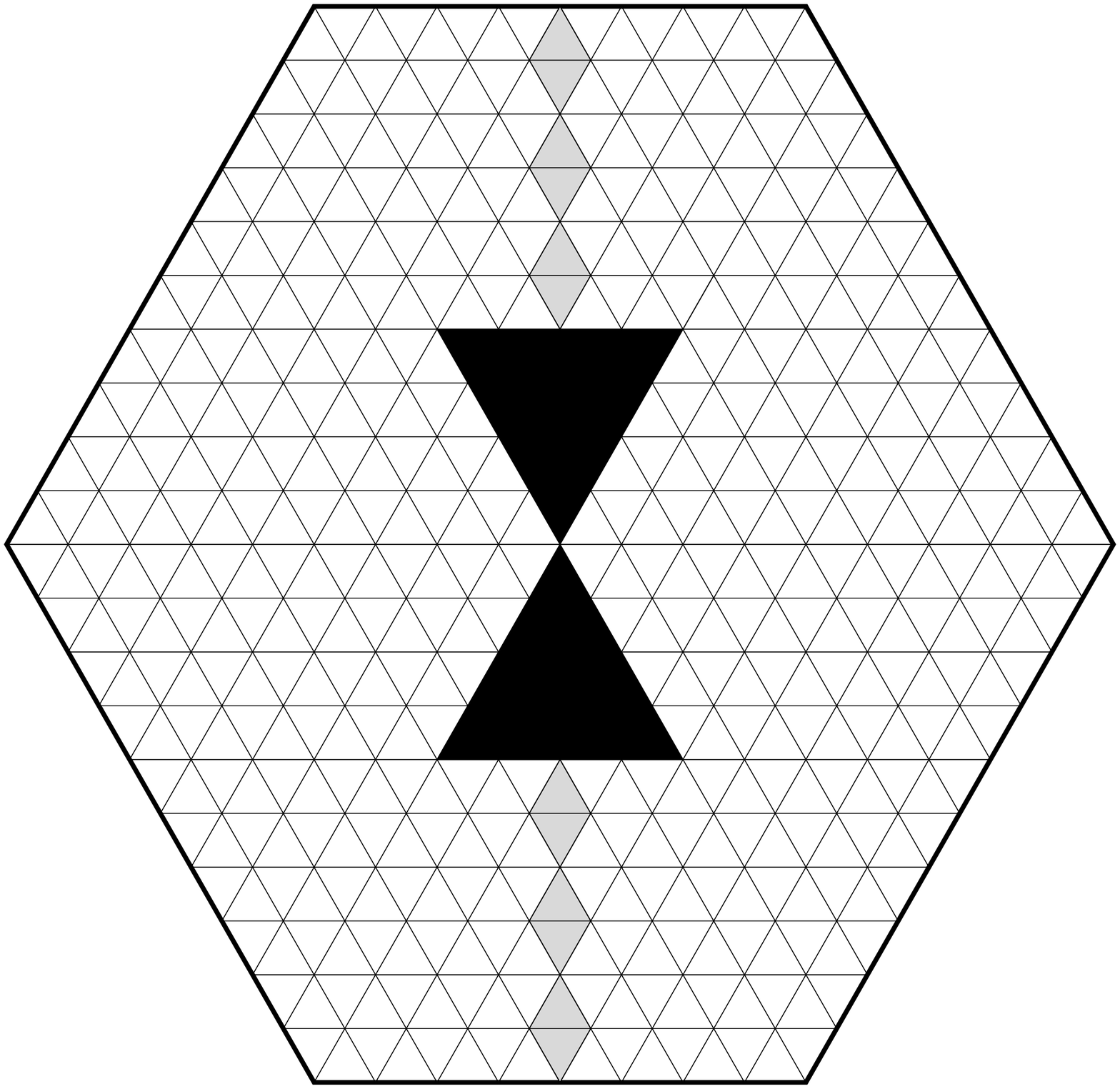}}
\hfill
}
\caption{\label{fbb} The $S$-cored hexagon $SC_{6,8,4}(3,1,2,2)$ (left; see \cite{vf} for details of its definition) and the region $H_{8,10}(4)$ (right).}
\label{fbb}
\end{figure}

\begin{align}
&\!\!\!\!\!\!
\M(G\setminus\{b,d\})\M_f(G\setminus\{a,c\})
\nonumber
=
\M(G\setminus\{a,d\})\M_f(G\setminus\{b,c\})
+
\M(G\setminus\{a,b\})\M_f(G\setminus\{c,d\}).
\nonumber
\end{align}

\medskip
Equation \eqref{ebe} also has a specialization that appeared before in Kuo's work: If $G$ is bipartite, has one more white than black vertices, $a$, $b$, $c$ are white and $d$ is black, then \eqref{ebe} becomes Kuo's Theorem 2.4 of \cite{KuoOne}:
\medskip
\begin{align}
&\!\!\!\!\!\!
\M(G\setminus\{b\})\M(G\setminus\{a,c,d\})
\nonumber
=
\M(G\setminus\{a\})\M(G\setminus\{b,c,d\})
+
\M(G\setminus\{a,b,d\})\M(G\setminus\{c\}).
\nonumber
\end{align}

\medskip
\parindent=0pt
More generally, if $G$ is not necessarily bipartite and $S=\emptyset$, \eqref{ebe} becomes
\medskip
\begin{align}
&\!\!\!\!\!\!
\M(G\setminus\{b\})\M(G\setminus\{a,c,d\})
+
\M(G\setminus\{d\})\M(G\setminus\{a,b,c\})
\nonumber
\\
&\ \ \ \ \ \ \ \ \ \ \ \ \ \ \ \ \ \ \ \ \ \ \ \ 
=
\M(G\setminus\{a\})\M(G\setminus\{b,c,d\})
+
\M(G\setminus\{a,b,d\})\M(G\setminus\{c\}).
\label{ebh}
\end{align}

\medskip
This counterpart of Kuo's 2006 theorem \eqref{ebg} seems to have gone unnoticed until now.

\medskip
\parindent=15pt
We state now the other main result of our paper, which deals with enumerating the ``symmetric and self complementary'' lozenge tilings of a hexagonal region with a shamrock removed from its center. This requires our tilings to be invariant under reflection across the horizontal, and under rotation by 180 degrees (i.e., central symmetry). Clearly, a necessary condition for the existence of such tilings is that the region itself is invariant under these symmetries. Central symmetry implies that the bottom lobes of the shamrock are empty, and that the top lobe is congruent to the central lobe. Thus the region is a hexagon $H_{x,y,k}$ of side-lengths $x$, $y$, $y$, $x$, $y$, $y$ (clockwise from top) with a vertical ``bowtie'' consisting of two triangles of side $k$ removed from its center (see the picture on the right in Figure \ref{fbb}).

Since horizontal symmetry and central symmetry imply vertical symmetry, and because in any vertically symmetric tiling the vertical symmetry axis must be entirely covered by vertical lozenges (see the shaded lozenges in Figure \ref{fbb}), it follows that all of $x$, $y$ and $k$ must be even if a tiling with the required symmetries exists.

\begin{figure}[h]
\centerline{
\hfill
{\includegraphics[width=0.44\textwidth]{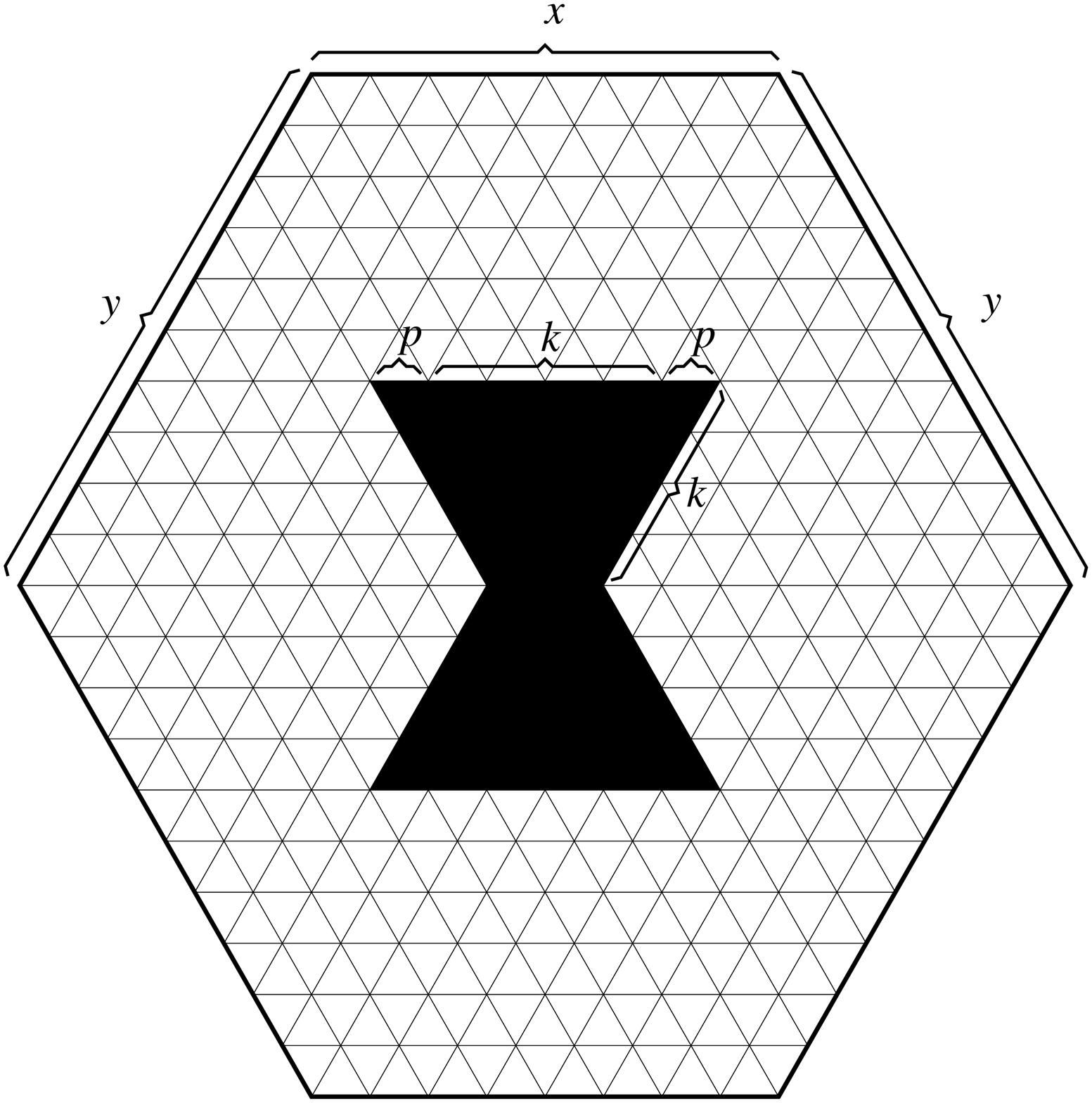}}
\hfill
{\includegraphics[width=0.30\textwidth]{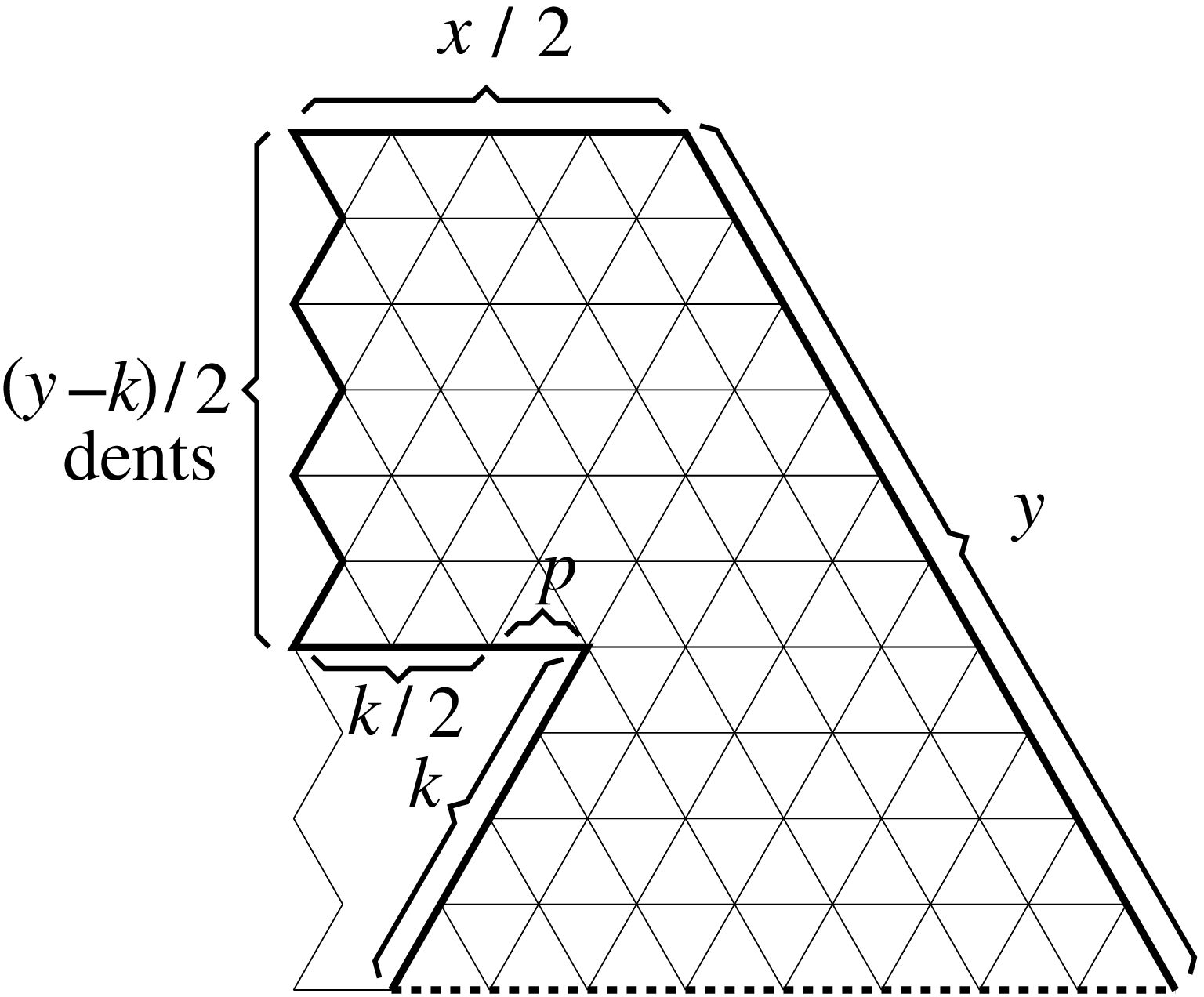}}
\hfill
}
\caption{The carpenter's butterfly region $H_{x,y}(k,p)$ for $x=8$, $y=10$, $k=4$, $p=1$ (left) and the flashlight region $F_{x/2,(y-k)/2,k/2,p}$ whose tilings can be identified with the horizontally and vertically symmetric tilings of the former (right).}
\label{fbc}
\end{figure}

\begin{figure}[h]
\centerline{
\hfill
{\includegraphics[width=0.25\textwidth]{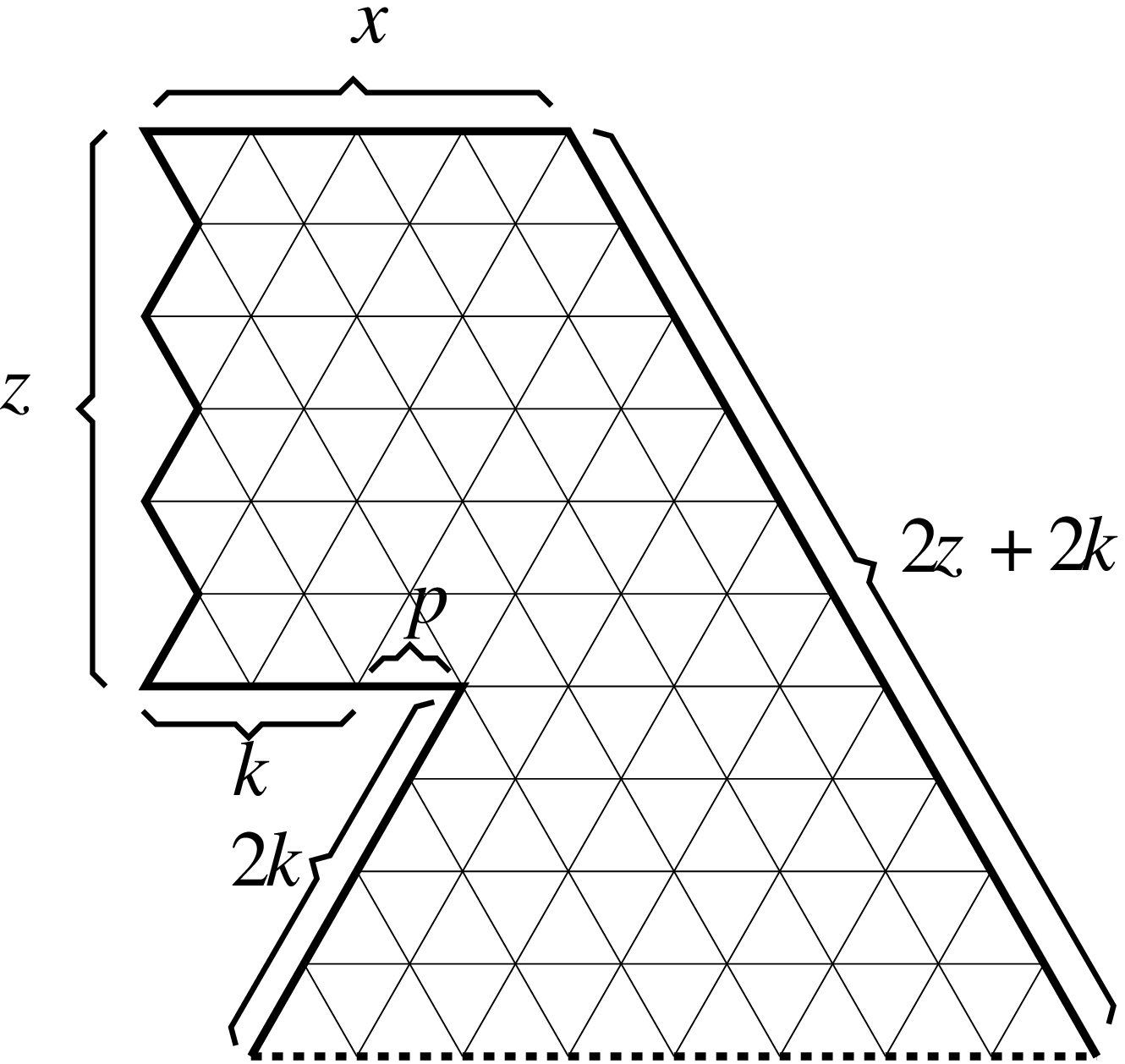}}
\hfill
{\includegraphics[width=0.30\textwidth]{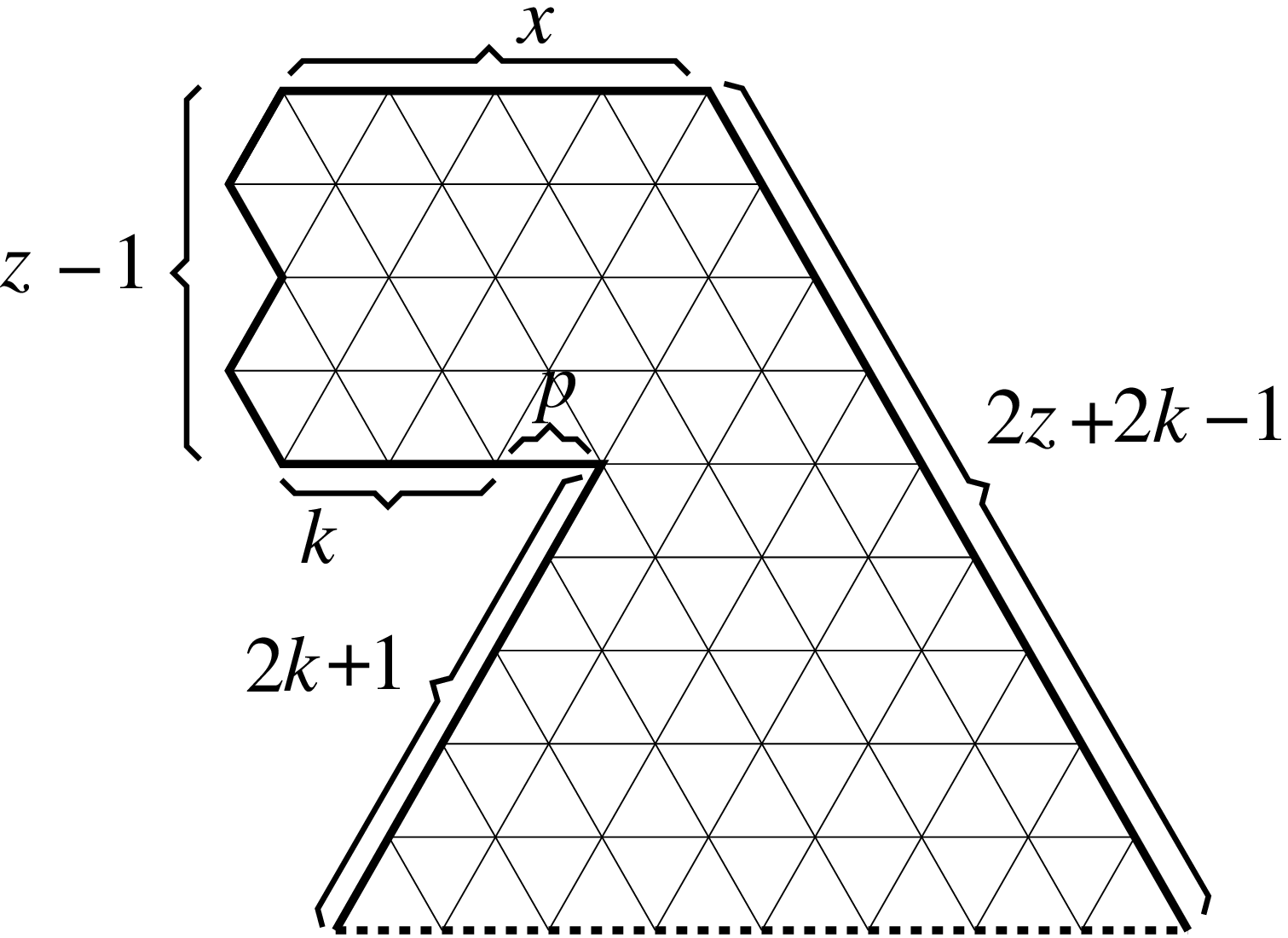}}
\hfill
}
\caption{The flashlight region $F_{x,z,k,p}$ (left) and the reduced flashlight region $\hat{F}_{x,z,k,p}$ (right) for $x=4$, $z=3$, $k=2$, $p=1$.}
\label{fca}
\end{figure}

Therefore, the symmetry case corresponding to symmetric and self complementary plane partitions amounts to enumerating horizontally and vertically symmetric lozenge tilings of the regions $H_{x,y}(k)$, where $x$, $y$ and $k$ are all even. 

In our main result we actually enumerate tilings of a more general family of regions. This extension turns out to be crucial in order for our proof to work. The more general regions are obtained by ``thickening'' the removed bowtie: Translate its left boundary $p$ units to the left, and its right boundary $p$ units to the right, turning the removed portion into a region resembling a carpenter's butterfly (we are assuming that the latter still fits inside the outer hexagon). Denote by $H_{x,y}(k,p)$ the resulting region (see Figure \ref{fbc} for an example).

The main enumeration result of this paper is the following.

\begin{figure}[h]
\centerline{
\hfill
{\includegraphics[width=0.40\textwidth]{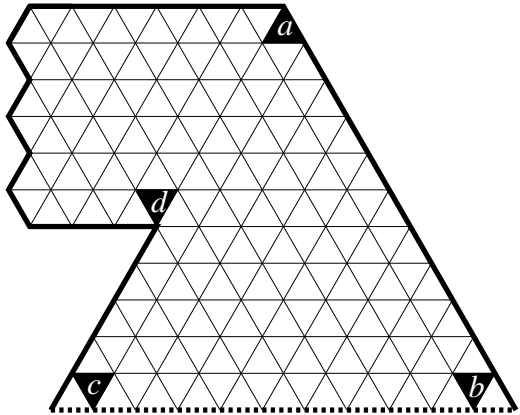}}
\hfill
}
\caption{\label{fbcc} Applying free boundary Kuo condensation to the region $\hat{F}_{x,z,k,p}$.}
\label{fcb}
\end{figure}

\begin{theo}
\label{tbc}
For any non-negative integers $x$, $y$, $k$ and $p$, we have 
%
\begin{align}
&
\M_{\,-\,,\,|\,}(H_{2x,2y}(2k,p))=
\prod_{i=1}^{y-k-1}\frac{k+i}{i}
\prod_{i=0}^{p-1}\frac{(x+y+p-2i)_{y-k-1}}
                      {(x+k+p-2i)_{y-k-1}}
\prod_{i=1}^{y-k-1}\prod_{j=2}^i\frac{2k+i+j-1}{i+j-1}
\nonumber
\\
&\ \ \ \ \ \ \ \ \ \ \ \ \ 
\times
\dfrac
{\prod_{j=1}^{k}(x-k-p+2j-1)_{2y+2k-4j+3}
\prod_{j=1}^{y-k}(x+k-p+j)_{2y-2k-2j+1}}
{\prod_{j=1}^{k}(2j-1)_{2y+2k-4j+3}
\prod_{j=1}^{y-k}(2k+j)_{2y-2k-2j+1}},
\label{ebi}
\end{align}
where all products for which the index limits are out of order are taken to be $1$.
\end{theo}

\section{Proof of Theorem \ref{tbc}}

Since, as we mentioned, each tiling with these symmetries must contain the vertical lozenges along the vertical symmetry axis (these are shaded in Figure \ref{fbb}), it follows that the horizontally and vertically symmetric lozenge tilings of $H_{x,y}(k,p)$ can be identified with tilings of the subregion of $H_{x,y}(k,p)$ that is to the right of the shaded vertical lozenges and above the horizontal symmetry axis, with the specification that the boundary along the horizontal symmetry axis is free, i.e. lozenges are allowed to protrude out halfway across it (the region obtained this way for the example on the left in Figure \ref{fca} is shown on the right in the same figure). Denote by $F_{x,z,k,p}$ the region of this type which has the dimensions indicated in Figure \ref{fca}, and call it a {\it flashlight region}. It is defined for all non-negative integers $x,z,k,p$ with $x+z\geq k+p$ (so that the dent on the lower left does not go through the boundary on the right).

Then we have
\begin{equation}
\M_{\,-\,,\,|\,}(H_{2x,2y}(2k,p))=\M_f(F_{x,y-k,k,p}),
\label{eca}
\end{equation}
and Theorem \ref{tbc} will follow from the following result.

\begin{figure}[h]
\vskip-0.1in
\centerline{
\hfill
{\includegraphics[width=0.35\textwidth]{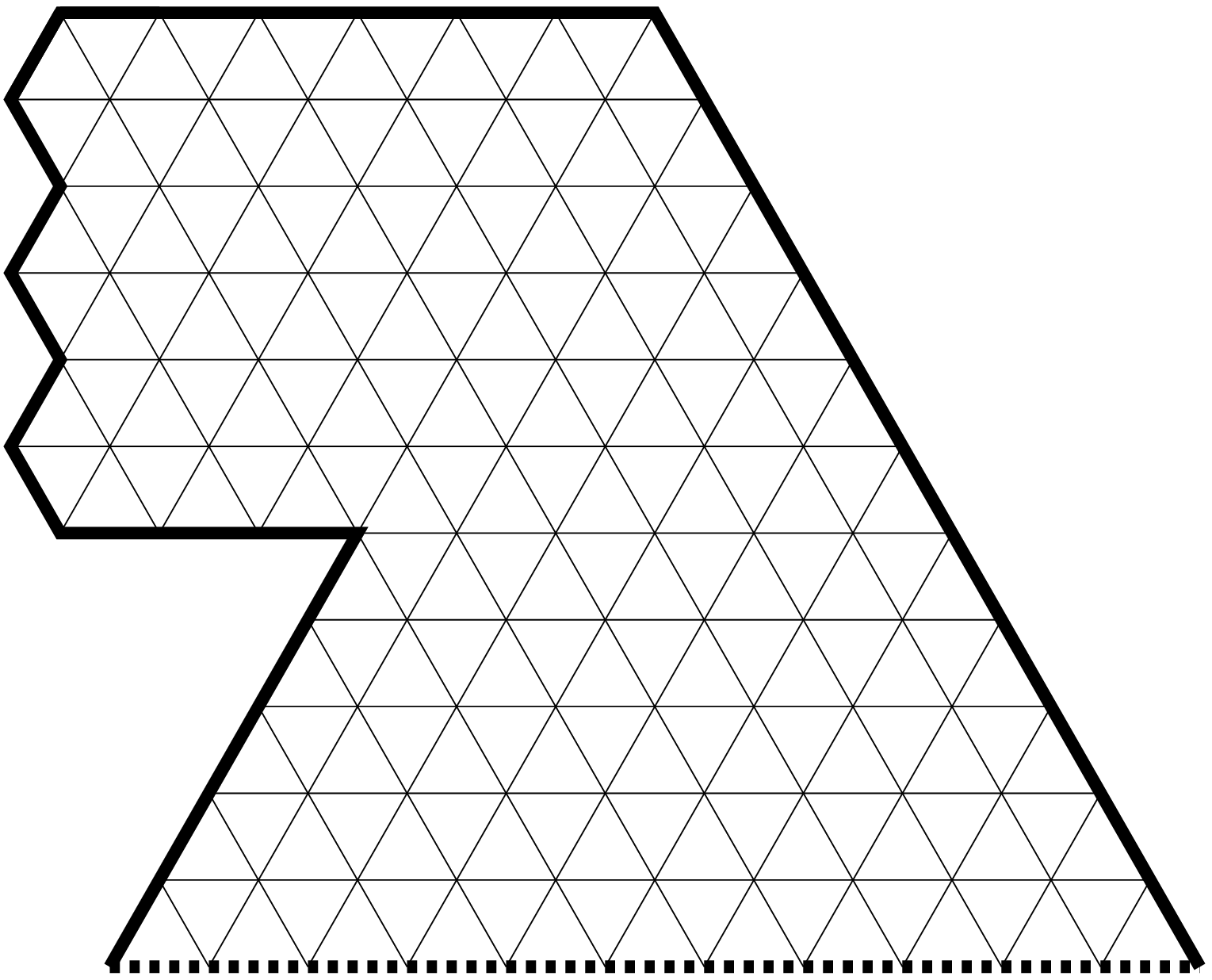}}
\hfill
{\includegraphics[width=0.35\textwidth]{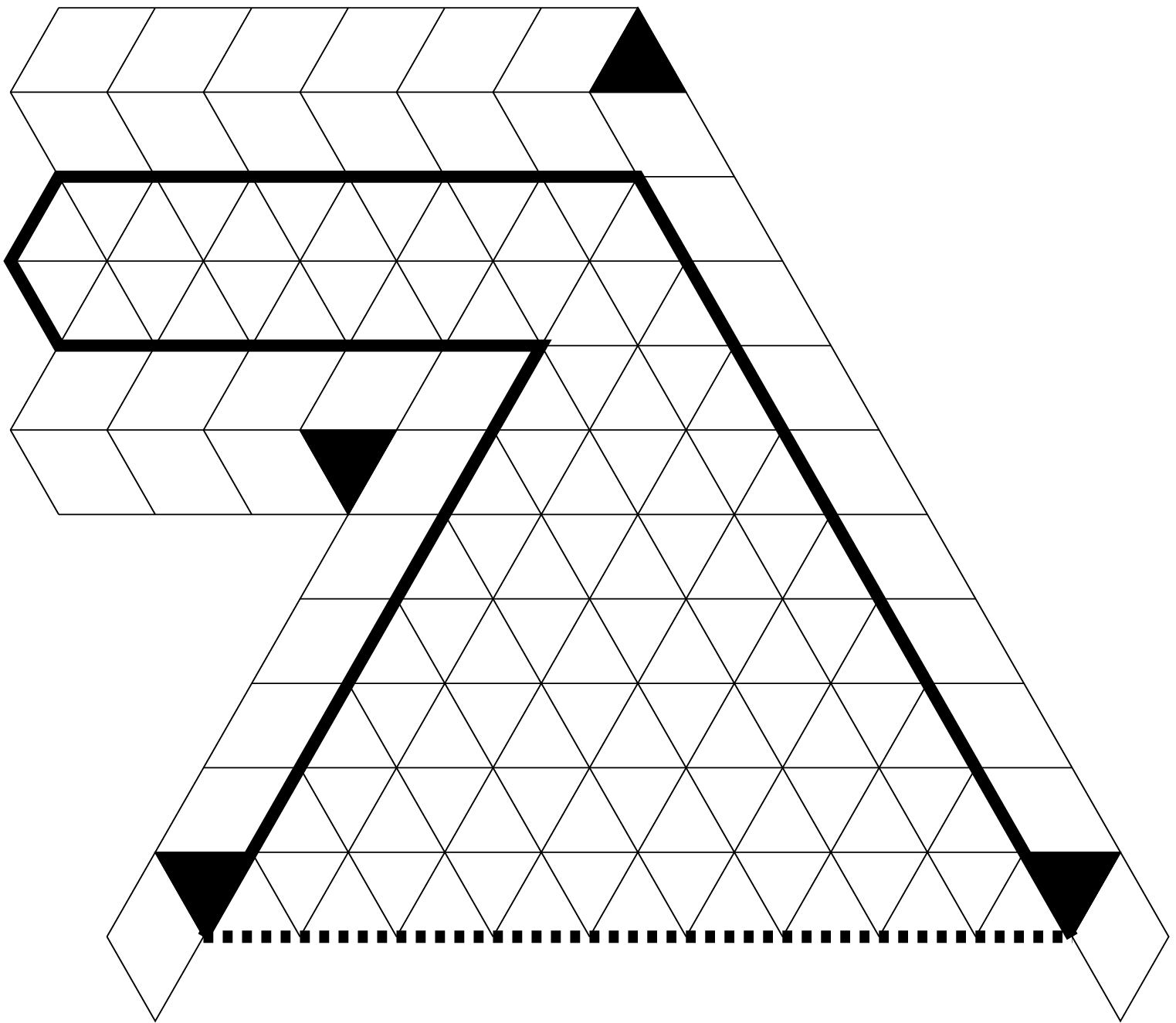}}
\hfill
}
\vskip0.1in
\centerline{
\hfill
{\includegraphics[width=0.35\textwidth]{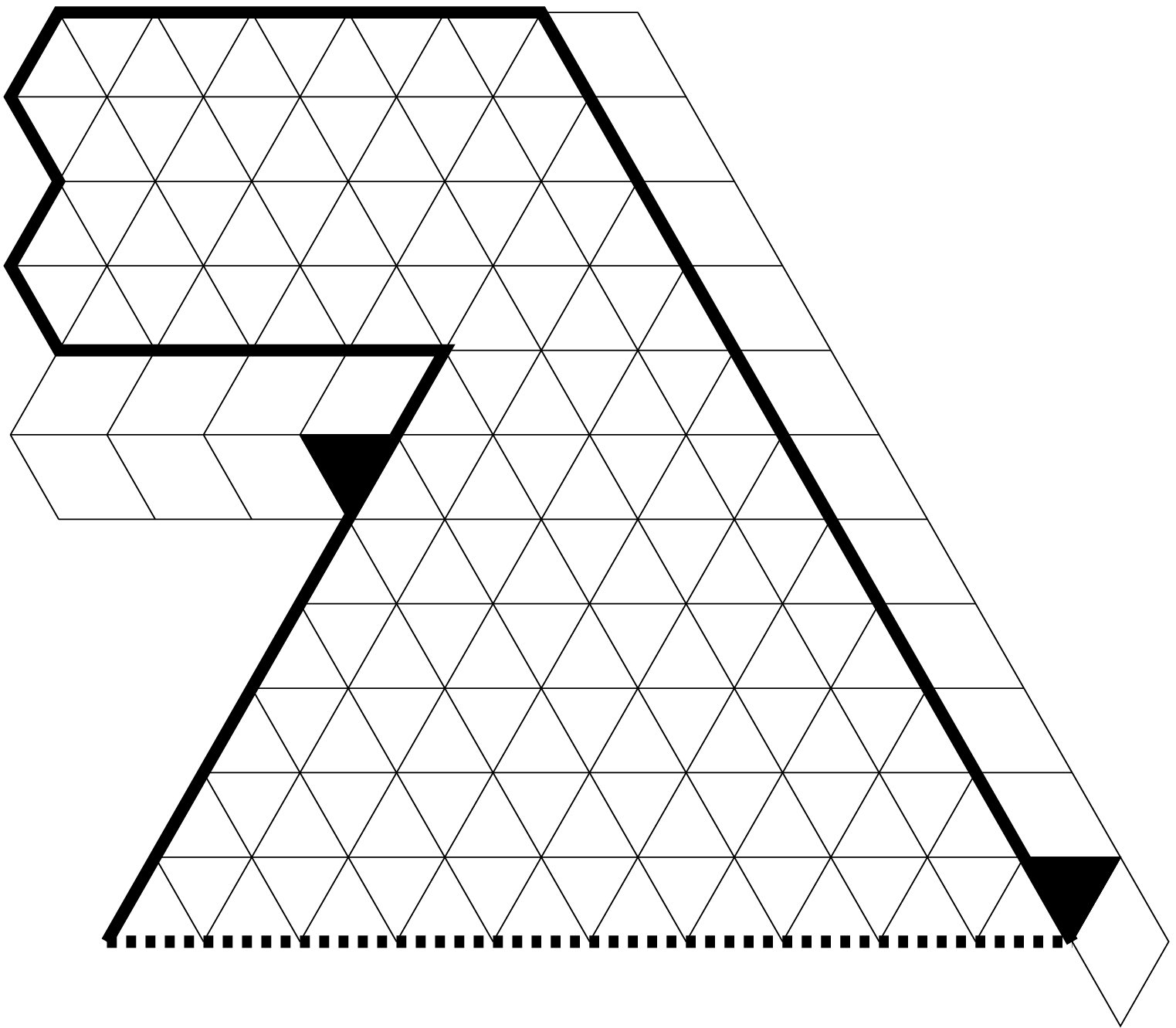}}
\hfill
{\includegraphics[width=0.35\textwidth]{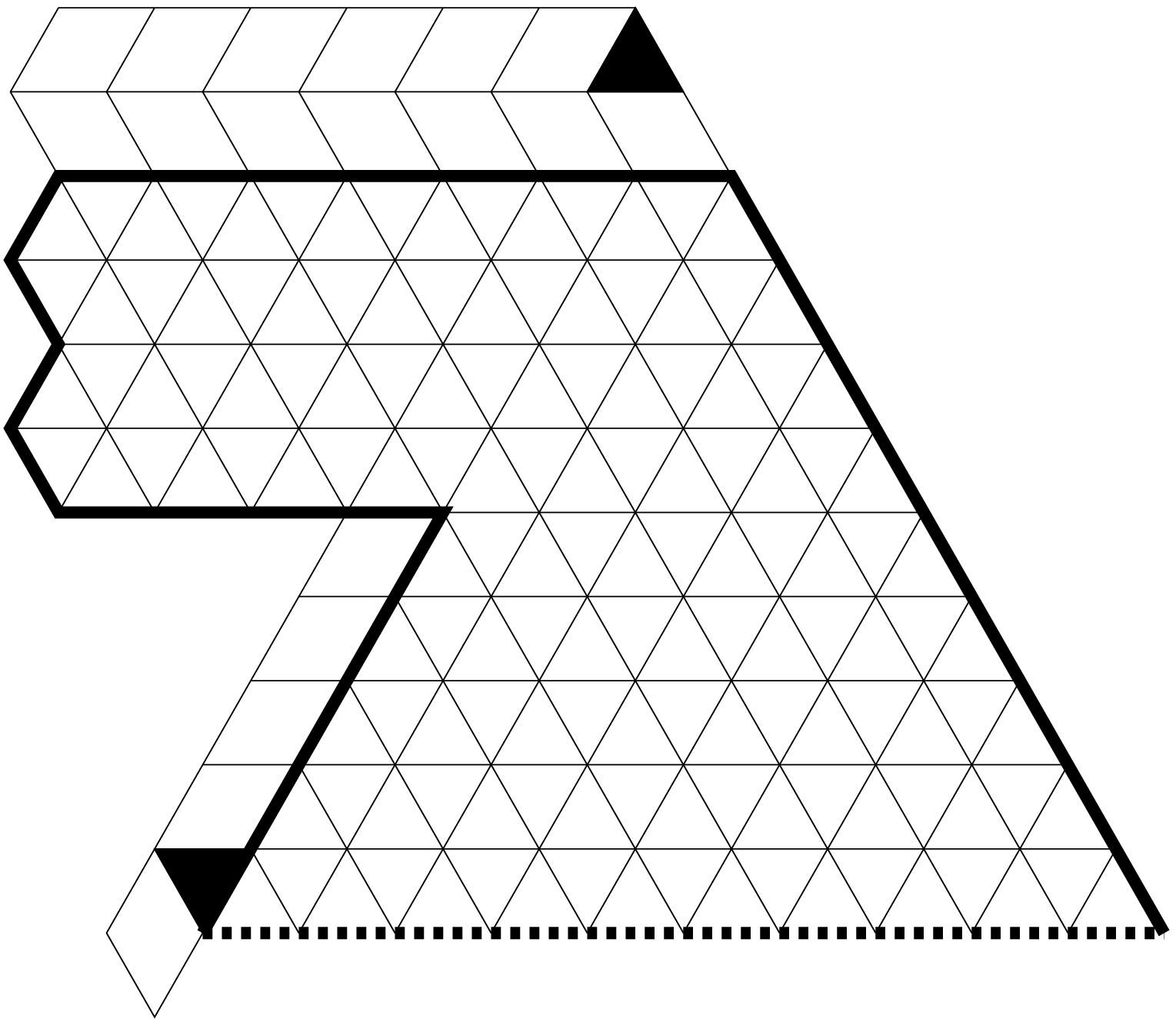}}
\hfill
}
\caption{The $\hat{F}$-regions on the left in \eqref{ecc}.}
\label{fcc}
\end{figure}

\begin{prop}
\label{tca}
For non-negative integers $x,z,k,p$ we have
%
\begin{align}
&
\M_f(F_{x,z,k,p})=
\prod_{i=1}^{z-1}\frac{k+i}{i}
\prod_{i=0}^{p-1}\frac{(x+z+k+p-2i)_{z-1}}
                      {(x+k+p-2i)_{z-1}}
\prod_{i=1}^{z-1}\prod_{j=2}^i\frac{2k+i+j-1}{i+j-1}
\nonumber
\\
&\ \ \ \ \ \ \ \ \ \ \ \ \ 
\times
\dfrac
{\prod_{j=1}^{k}(x-k-p+2j-1)_{2z+4k-4j+3}
\prod_{j=1}^{z}(x+k-p+j)_{2z-2j+1}}
{\prod_{j=1}^{k}(2j-1)_{2z+4k-4j+3}
\prod_{j=1}^{z}(2k+j)_{2z-2j+1}},
\label{ecb}
\end{align}
where all products for which the index limits are out of order are taken to be $1$.

\end{prop}

\begin{figure}[h]
\centerline{
\hfill
{\includegraphics[width=0.35\textwidth]{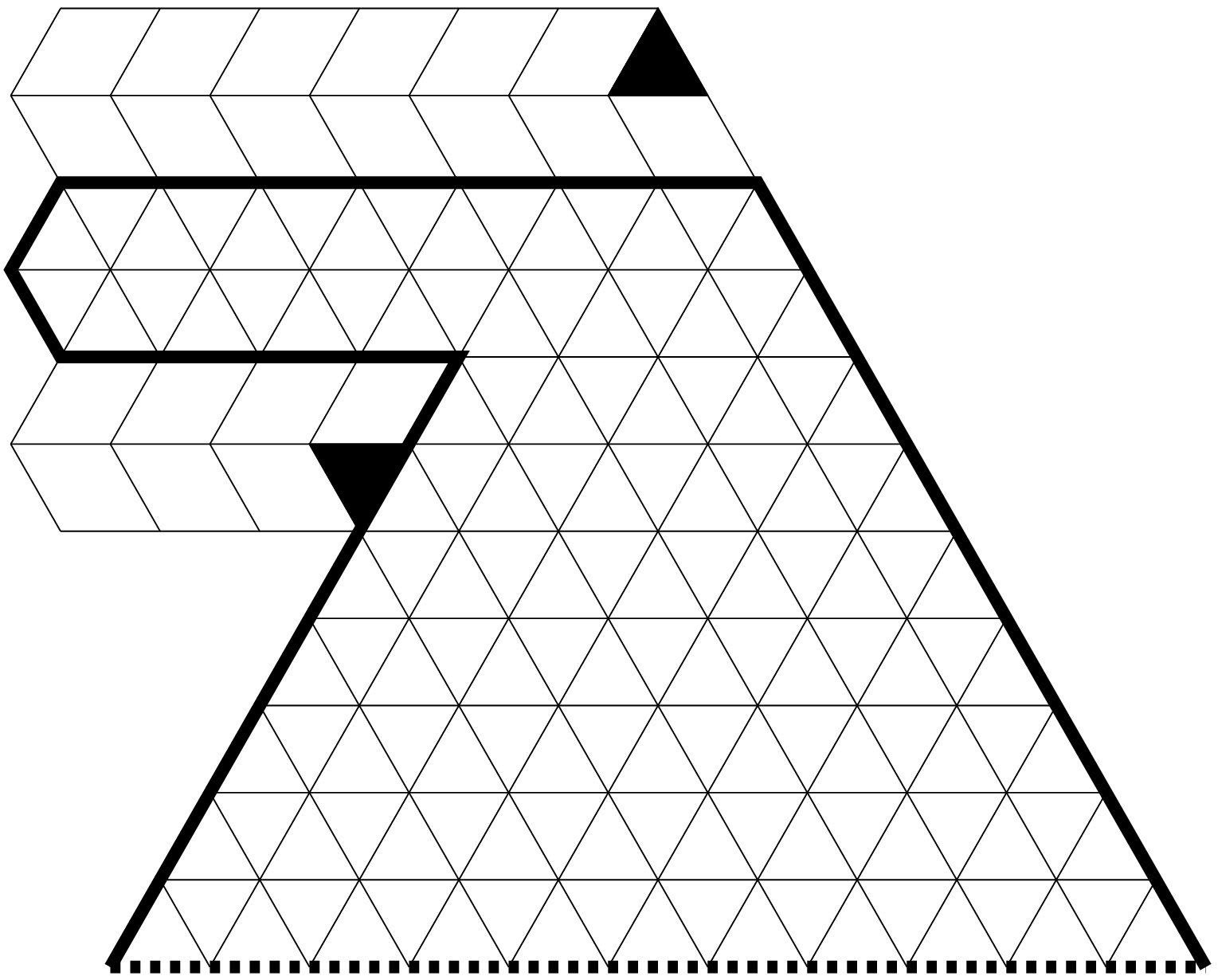}}
\hfill
{\includegraphics[width=0.35\textwidth]{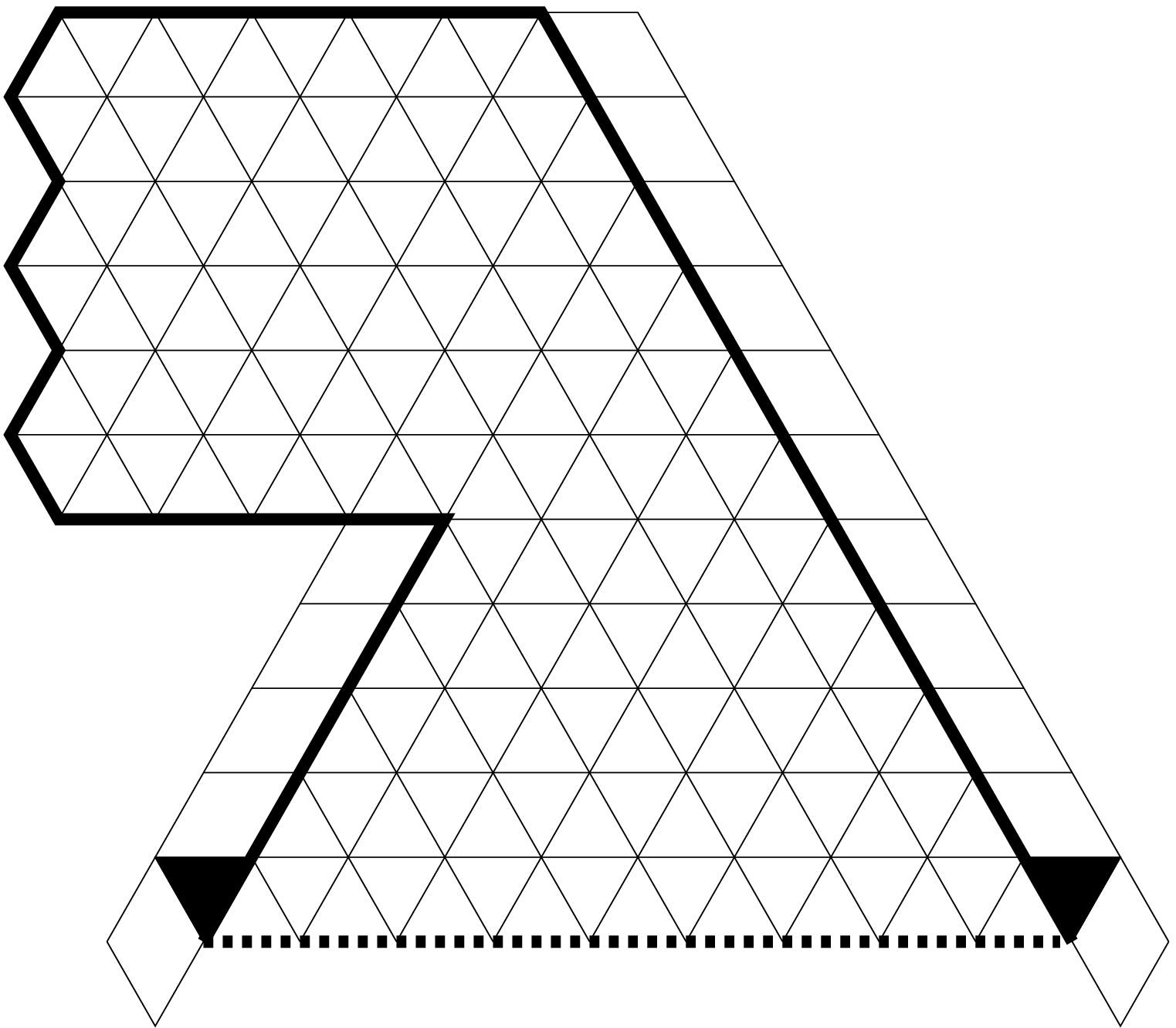}}
\hfill
}
\vskip0.1in
\centerline{
\hfill
{\includegraphics[width=0.35\textwidth]{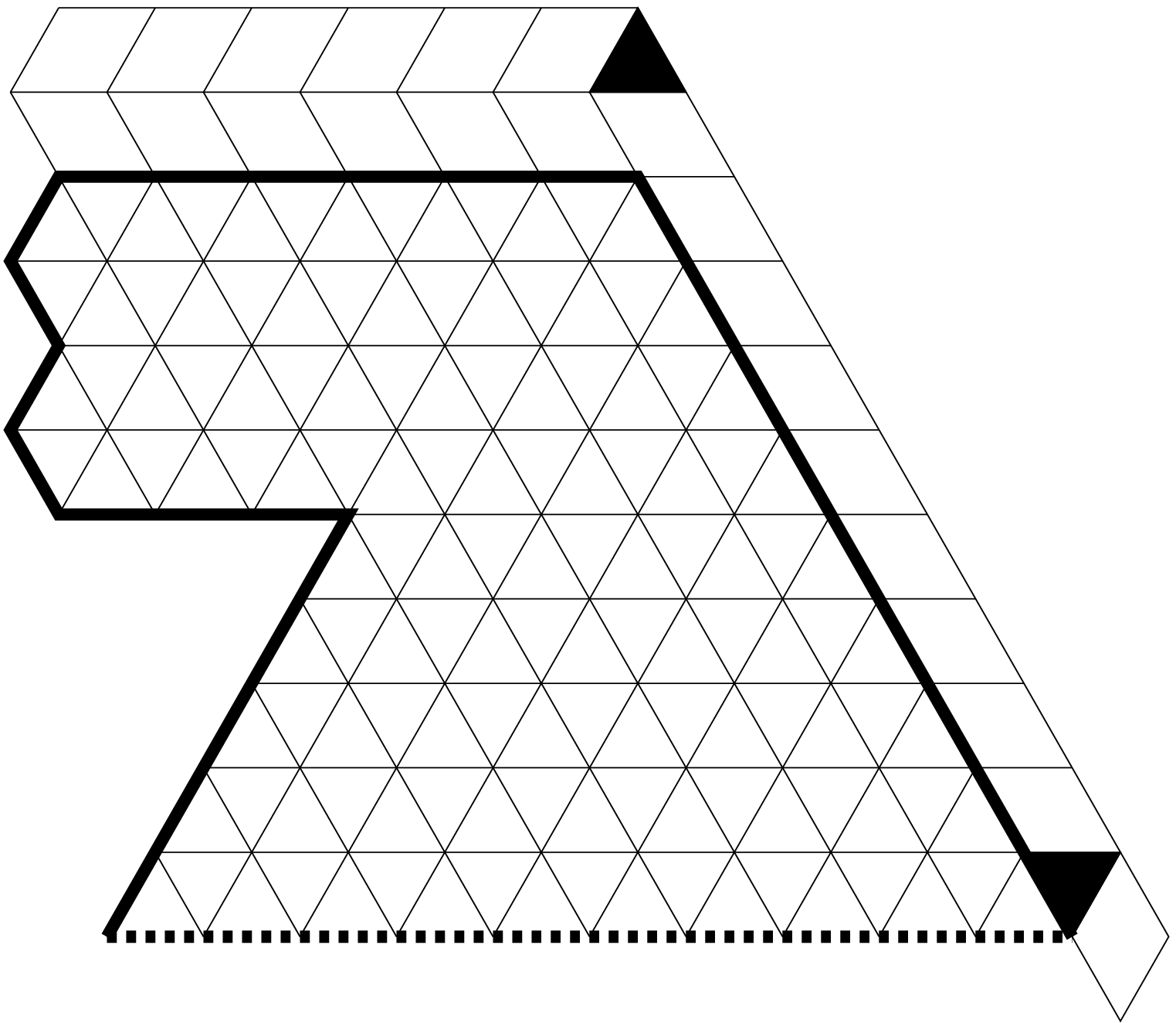}}
\hfill
{\includegraphics[width=0.35\textwidth]{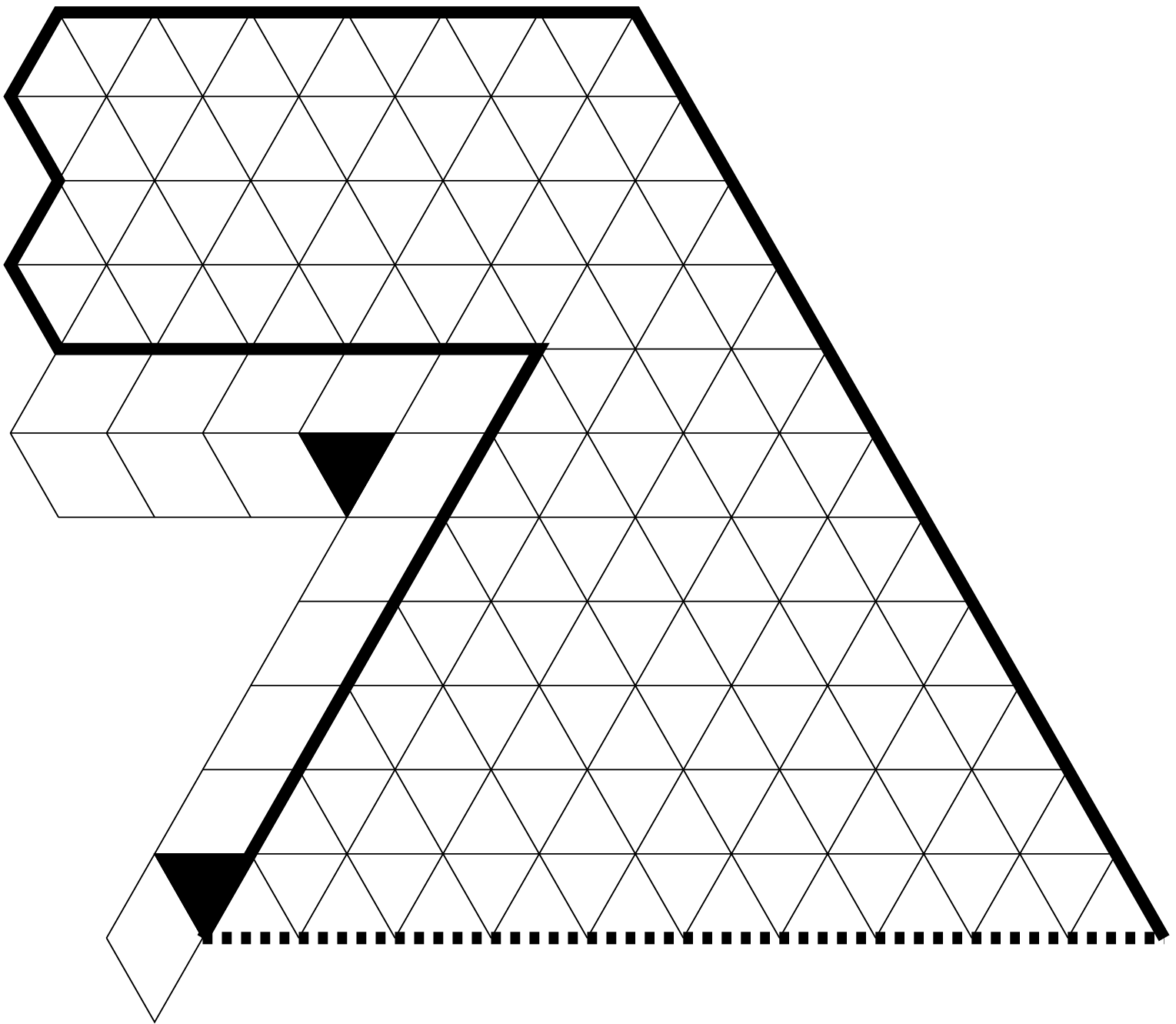}}
\hfill
}
\caption{The $F$-regions on the right in \eqref{ecc}.}
\label{fcd}
\end{figure}

\begin{proof} Note that if $z>0$, the zig-zag boundary portion of $F_{x,z,k,p}$ is non-empty, and $x$ lozenges along the top as well as $k+p$ lozenges just above the horizontal portion of the notch are forced. Denote by $\hat{F}_{x,z,k,p}$ the region obtained from $F_{x,z,k,p}$ after removing these forced lozenges (see Figure \ref{fca}). Then we clearly have

\begin{equation}
\M_f(F_{x,z,k,p})=\M_f(\hat{F}_{x,z,k,p}).
\label{ecbb}
\end{equation}

Identify the region $\hat{F}_{x,z,k,p}$ with its planar dual graph. Note that if we choose $a,b,c,d$ as indicated in Figure \ref{fcb}, then the set $S$ of free vertices (which correspond to the up-pointing unit triangles resting on the bottom dotted boundary) is both $a,c$-separated and $b,d$-separated. Therefore we can apply Corollary \ref{tbb}. When we do so, all the subregions obtained by removing from $\hat{F}_{x,z,k,p}$ the subsets of $\{a,b,c,d\}$ that show up in \eqref{ebd} turn out to be, after removing forced lozenges, flashlight regions of various arguments (see Figures \ref{fcc} and \ref{fcd}). We therefore obtain from \eqref{ebd} that

\begin{align}
&
\M_f(\hat{F}_{x,z,k,p})\M_f(\hat{F}_{x,z-2,k+1,p+1})
+
\M_f(\hat{F}_{x-1,z-1,k+1,p})\M_f(\hat{F}_{x+1,z-1,k,p+1})
\nonumber
\\
&\ \ \ \ \ \ \ \ \ \ \ \ \ \ 
=
\M_f(\hat{F}_{x+1,z-2,k+1,p})\M_f(\hat{F}_{x-1,z,k,p+1})
+
\M_f(\hat{F}_{x,z-1,k,p})\M_f(\hat{F}_{x,z-1,k+1,p+1}),
\label{ecc}
\end{align}
which gives, using \eqref{ecbb}, that

\begin{figure}[h]
\centerline{
\hfill
{\includegraphics[width=0.15\textwidth]{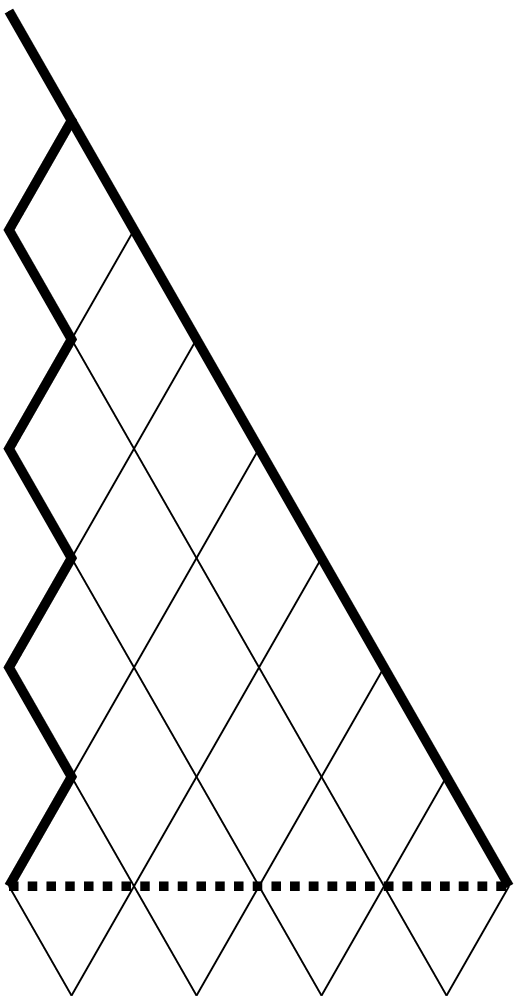}}
\hfill
{\includegraphics[width=0.35\textwidth]{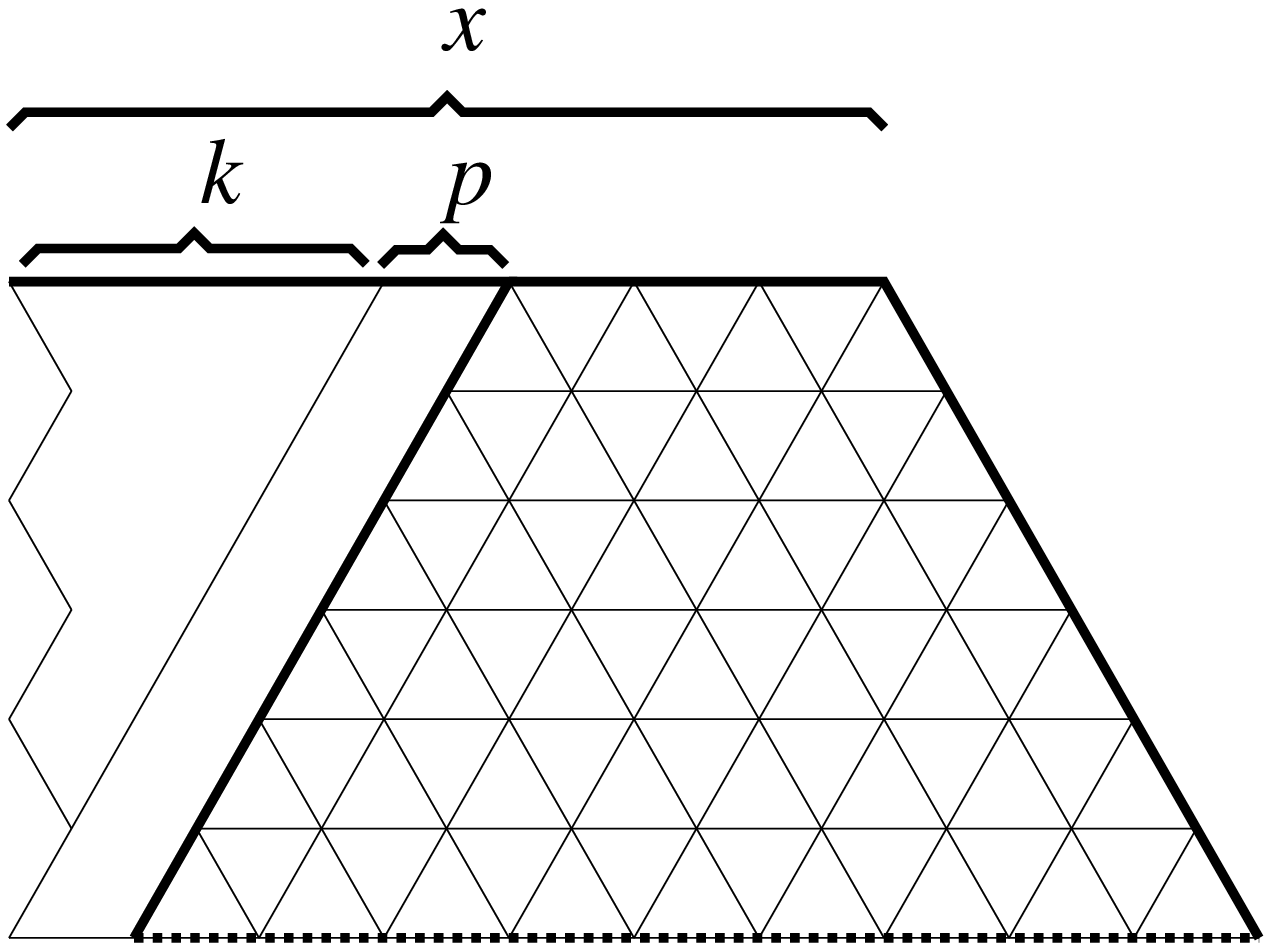}}
\hfill
}
\caption{The base cases $x=0$ (left) and $z=0$ (right).}
\label{fce}
\end{figure}

\begin{figure}[h]
\centerline{
\hfill
{\includegraphics[width=0.35\textwidth]{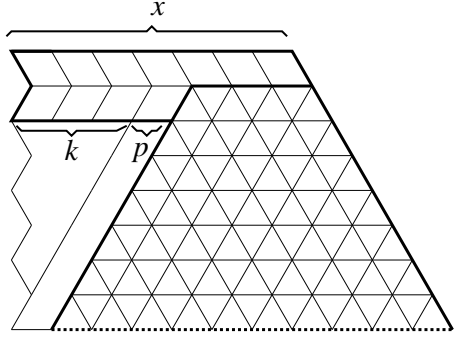}}
\hfill
}
\caption{The base case $z=1$.}
\label{fcf}
\end{figure}

\begin{align}
&
\M_f(F_{x,z,k,p})\M_f(F_{x,z-2,k+1,p+1})
+
\M_f(F_{x-1,z-1,k+1,p})\M_f(F_{x+1,z-1,k,p+1})
\nonumber
\\
&\ \ \ \ \ \ \ \ \ \ \ \ \ \ 
=
\M_f(F_{x+1,z-2,k+1,p})\M_f(F_{x-1,z,k,p+1})
+
\M_f(F_{x,z-1,k,p})\M_f(F_{x,z-1,k+1,p+1}).
\label{ecd}
\end{align}
%

%
%


Equation \eqref{ecd} holds for all integers $x\geq1$, $z\geq2$ and $k,p\geq0$ (strictly speaking, we assumed in Figures \ref{fcc} and \ref{fcd} that $z-1\geq2$; however, one readily sees by considering the analogous pictures for $z-1=1$ that the resulting regions in this case also lead to \eqref{ecd}).

We prove \eqref{ecb} by induction on $x+2z$. This works because for the eight $F$-regions in  \eqref{ecd}, the value of this statistic is $x+2z$ for first region, and strictly less for all the others. View therefore \eqref{ecd} as a recurrence giving the number of tilings of the first region in terms of the others.

For the $F$-regions in \eqref{ecd} other than $F_{x,y,k,p}$, the value of the statistic $x+2z$ is either one, two, three or four units less than the value for $F_{x,y,k,p}$. On the other hand, in order for all the regions involved in \eqref{ecc} to be defined, one needs $x\geq1$ and $z\geq2$. Therefore, the base cases of our induction are the situations when $x+2z=i$, $i\in\{0,1,2,3\}$, and the additional three cases $x=0$, $z=0$ and $z=1$. Since $x$ and $z$ are non-negative integers, $x+2z\leq3$ implies $z=0$ or $z=1$. Thus it is enough to check the base cases $x=0$, $z=0$ and $z=1$.

Before we address these base cases, note that due to the Pochhammer symbols at the numerator in the second line of \eqref{ecb}, the expression on the right hand side in \eqref{ecb} is equal to zero if $x<k+p$. This proves \eqref{ecb} in this case, as if $x<k+p$, the $k+p$ paths of lozenges that would start upward from the side of length $k+p$ in any tiling of $F_{x,z,k,p}$ would not have enough room to end on the top side.

Assume now that $x=0$. If $k+p>0$, \eqref{ecb} follows from the above observation. In the remaining case of $x=k=p=0$, the region $F_{0,z,0,0}$ looks as shown on the left in Figure \ref{fce}. All tiles are forced, so $\M_f(F_{0,z,0,0})=1$, which agrees with the $x=k=p=1$ specialization of the expression on the right hand side of \eqref{ecb}.

Consider now the base case $z=0$. The region $F_{x,0,k,p}$ is as shown on the right in Figure \ref{fce}. As we saw above, we may assume that $x\geq k+p$. Then it follows from the figure that
\begin{equation}
\M_f(F_{x,0,k,p})=SPP(2k,2k,x-k-p),
\label{ece}
\end{equation}
where $SPP(a,a,b)$ is the number of symmetric plane partitions that fit in an $a\times a \times b$ box, given by MacMahon's formula (proved by Andrews \cite{AndSPP})
\begin{equation}
SPP(a,a,b)=\prod_{i=1}^a\left[\frac{b+2i-1}{2i-1}\prod_{j=i+1}^a\frac{b+i+j-1}{i+j-1}\right].
\label{ecee}
\end{equation}

Then \eqref{ecb} follows by the fact that the right hand side of \eqref{ece} (given by the above formula) agrees with the $z=0$ specialization of the expression on the right hand side of \eqref{ecb}.

For the last base case, $z=1$, the region $F_{x,1,k,p}$ looks as pictured in Figure \ref{fcf}. Upon removing the forced lozenges, the leftover region is a trapezoid of side-lengths $2k+1$, $x-k-p-1$, $2k+1$, with free boundary along its base. Thus, 
\begin{equation}
\M_f(F_{x,1,k,p})=SPP(2k+1,2k+1,x-k-p-1),
\label{ecf}
\end{equation}
and \eqref{ecb} follows by the fact that the right hand side of \eqref{ecf} (given by \eqref{ecee}) agrees with the $z=0$ specialization of the expression on the right hand side of \eqref{ecb}. 

For the induction step, assume that \eqref{ecb} holds for all $F$-regions for which the value of the $x$-parameter plus twice the value of the $z$-parameter is strictly less than $x+2z$. Use equation~\eqref{ecd} to express $\M_f(F_{x,z,k,p})$ in terms of $\M_f(F_{x',z',k',p'})$'s with $x'+2z'<x+2z$. By the induction hypothesis, all the involved $\M_f(F_{x',z',k',p'})$'s are given by formula \eqref{ecb}. It is routine to check that the resulting formula for $\M_f(F_{x,z,k,p})$ agrees with the expression on the right hand side of~\eqref{ecb}. This completes the proof. \end{proof}

%
%

\section{Asymptotics --- promontory in constrained/free corner}

If $k$ and $p$ are fixed while $x$ and $z$ grow to infinity, the region $F_{x,z,k,p}$ becomes an infinite~90 degree wedge with constrained boundary along the vertical zig-zag boundary portion and free boundary along the horizontal lattice line boundary portion. Our formulas allow us to find the answer to the following natural question: What is the effect of the presence of the dent (``promontory in the sea of dimers'') in the corner?

In view of the relationship of the flashlight regions $F_{x,y-k,k,p}$ to the carpenter's butterfly regions $H_{2x,2y}(k,p)$ (see Figure \ref{fbc}), so as to not distort the dimer statistics, we take $x=y$ as the boundary is sent to infinity. Therefore the question is to determine the correlation $\omega_c(k,p)$ of the dent with the corner, defined by

\begin{equation}
\omega_c(k,p):=\lim_{x\to\infty}\frac{\M_f(F_{x,x-k,k,p})}{\M_f(F_{x,x,0,0})}
\label{eda}
\end{equation}
(the regions at the numerator and denominator in the above limit are shown on the left in Figure \ref{fda} for $x=10$, $k=2$, $p=1$).

\begin{figure}[h]
\centerline{
\hfill
{\includegraphics[width=0.25\textwidth]{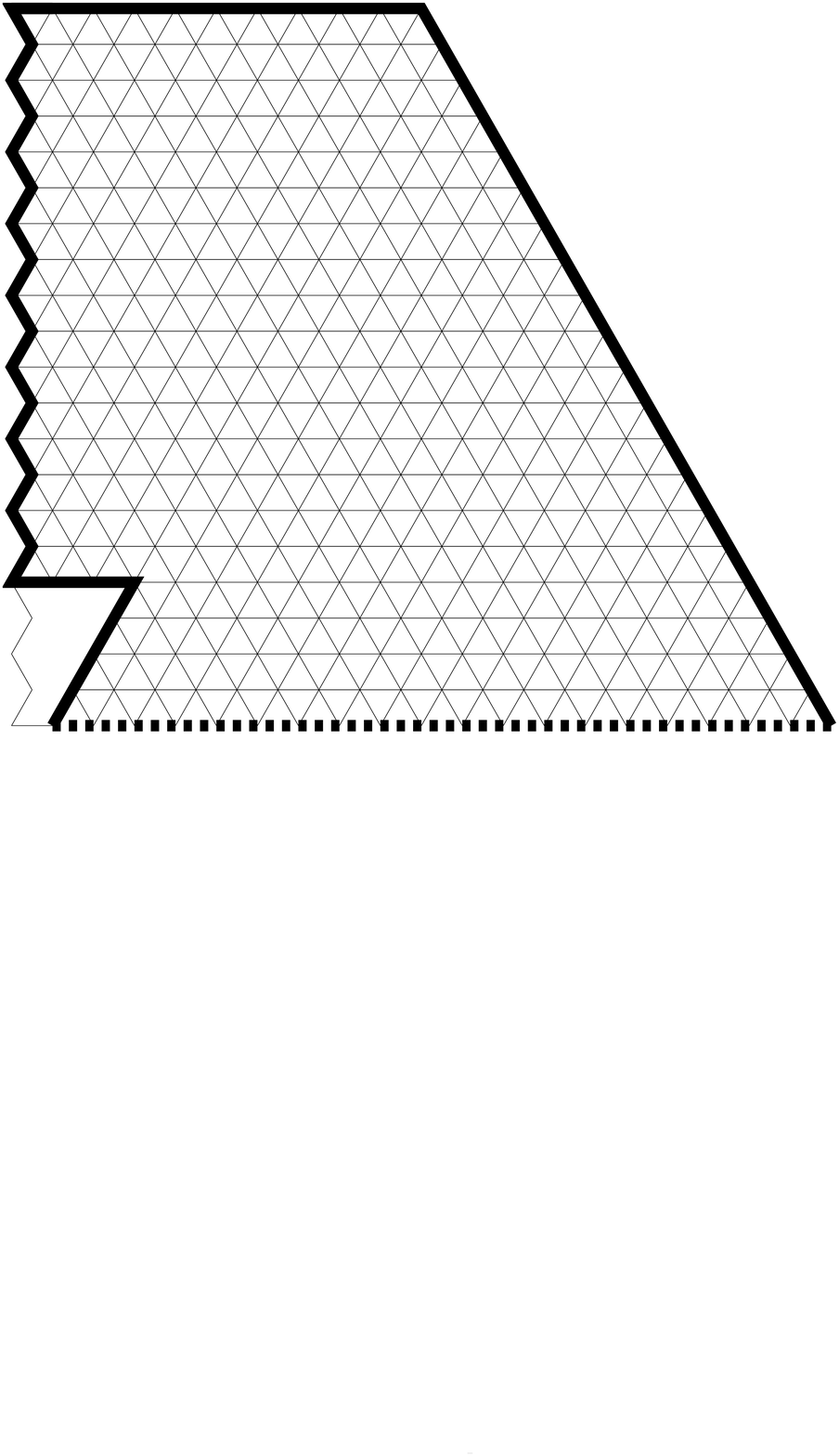}}
\hfill
{\includegraphics[width=0.50\textwidth]{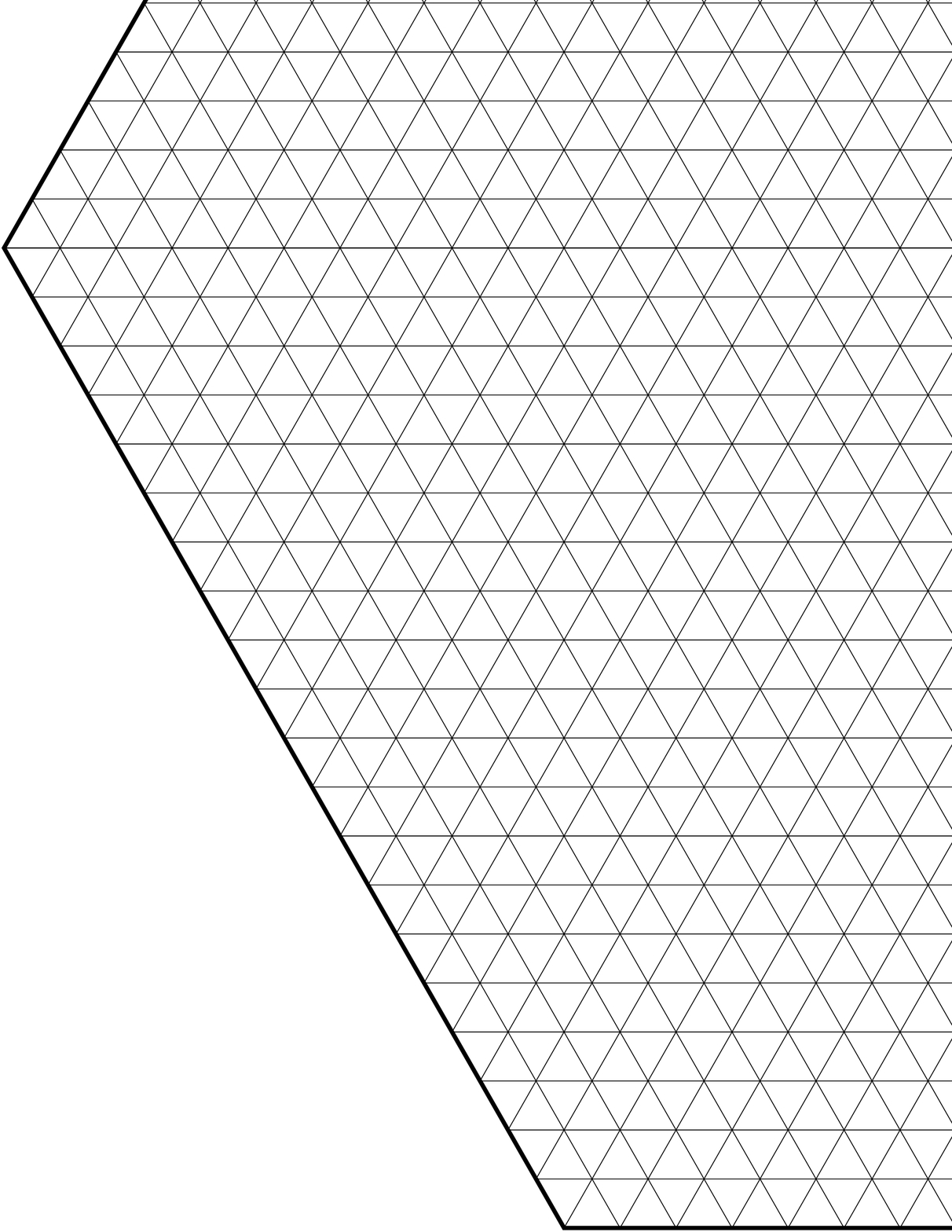}}
\hfill
}
\vskip0.1in
\centerline{
\hfill
{\includegraphics[width=0.25\textwidth]{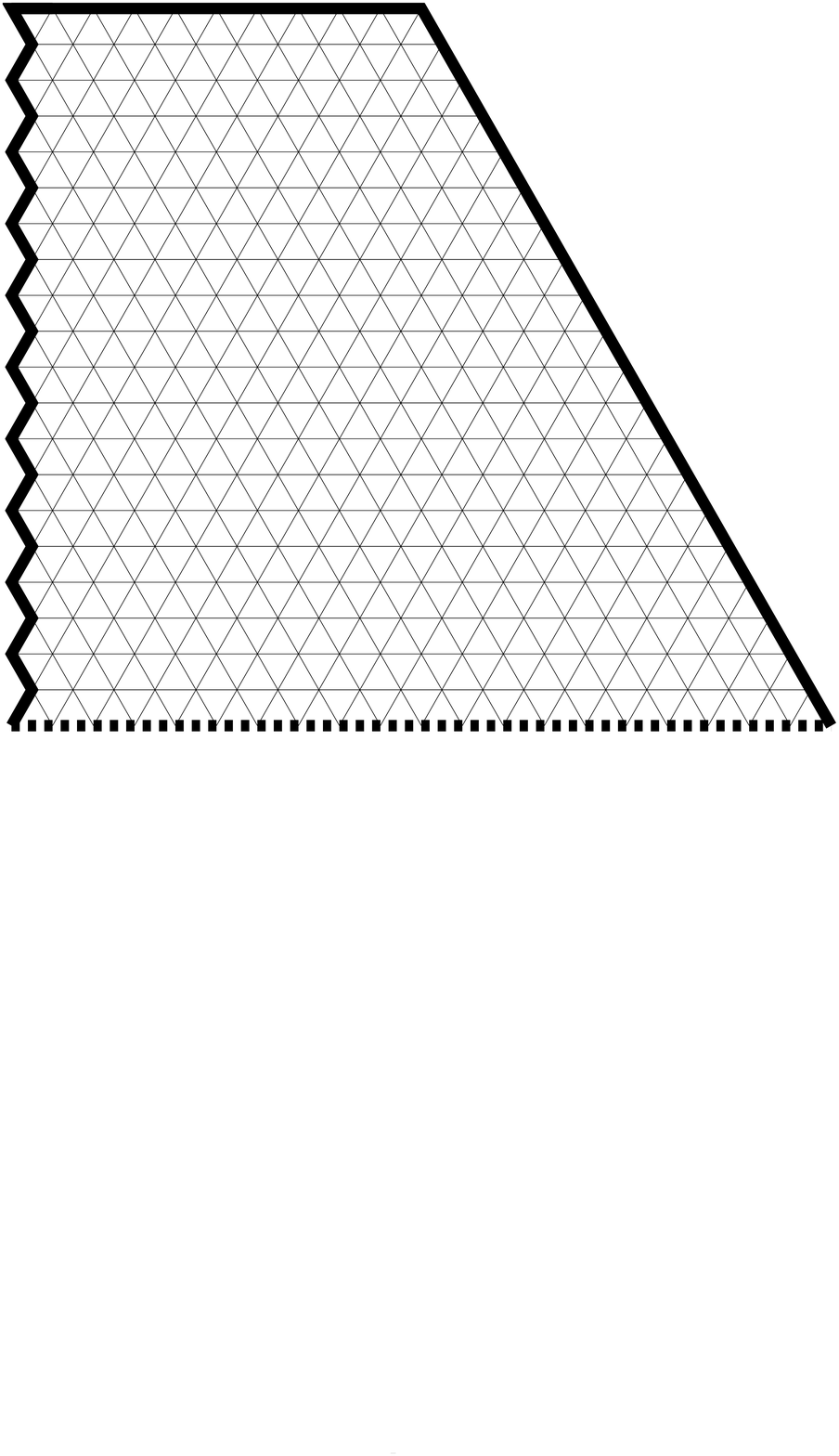}}
\hfill
{\includegraphics[width=0.50\textwidth]{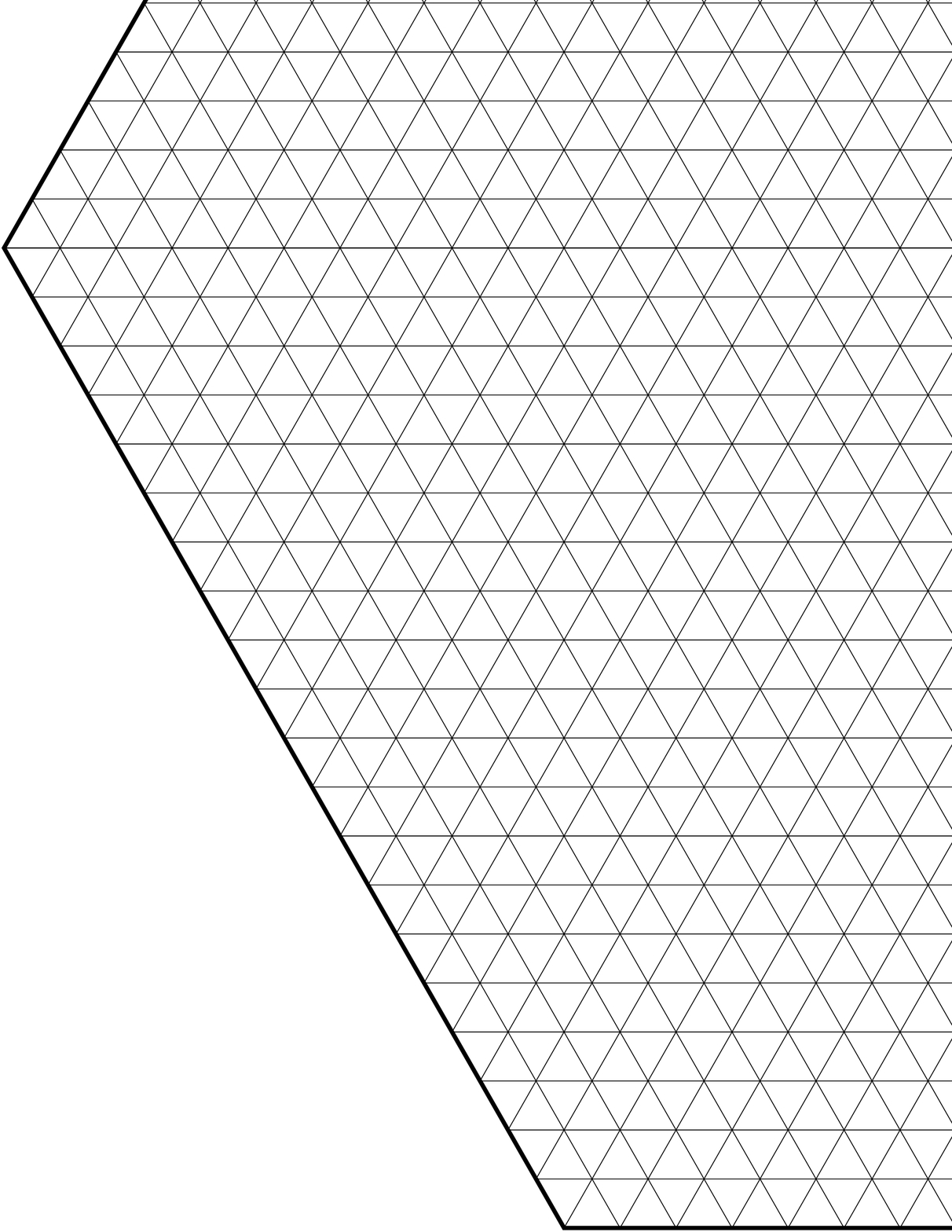}}
\hfill
}
\caption{The regions $F_{10,8,2,1}$ and $F_{10,10,0,0}$ (left) and $H_{10,10}(2,1)$ and $H_{10,10}(0,0)$ (right).}
\label{fda}
\end{figure}

The set-up is very similar to the one in our previous work \cite{rangle}, where instead of a dent (of shape and size depending on $k$ and $p$) in the corner, we had a single triangular hole of side-length 2 in the interior of the wedge at given distances from the two sides of the wedge. We saw in \cite{rangle} that the limit analogous to \eqref{eda} had, up to a multiplicative constant, the same asymptotics as the fourth root of of the correlation of the ``symmetrized system'' --- the system obtained by reflecting the 90 degree wedge in the two sides (thus ending up in that case with four triangular holes of side 2), obtaining the whole plane and eliminating the boundaries.

It would therefore be interesting to compare the $k,p\to\infty$ asymptotics of $\omega_c(k,p)$ with that of the fourth root of the bulk correlation $\omega(k,p)$ of the carpenter's butterfly, defined by 
\begin{equation}
\omega(k,p):=\lim_{x\to\infty}\frac{\M(H_{2x,2x}(k,p))}{\M(H_{2x,2x}(0,0))}
\label{edb}
\end{equation}
(the regions at the numerator and denominator in the above limit are shown on the right in Figure \ref{fda} for $x=10$, $k=2$, $p=1$).

The results of this section give explicit expressions for the correlations $\omega_c(k,p)$ and $\omega(k,0)$.

\begin{theo}
\label{tda}
For non-negative integers $k$ and $p$ we have
\begin{equation}
\omega_c(k,p)=\frac
{3^{\frac12 k^2-kp+\frac12 p^2+\frac{k}{2}-\frac{p}{2}}}
{2^{2k^2+k+p^2}}.
\label{edc}
\end{equation}
\end{theo}

\begin{proof} 
Using the formula from Proposition \ref{tbb}, we get
\begin{equation}
\frac{\M_f(F_{x,x-k,k,p})}{\M_f(F_{x,x,0,0})}
=
P_1P_2P_3P_4\frac{P_5}{\left(P_5|_{k=0,p=0}\right)},
\label{edca}
\end{equation}
where
\begin{align}
P_1&=\prod_{i=1}^{x-k-1}\frac{k+i}{i}
\label{edcb}
\\
P_2&=\prod_{i=0}^{p-1}\frac{(2x+p-2i)_{x-k-1}}{(x+k+p-2i)_{x-k-1}}
\label{edcc}
\end{align}
\begin{align}
P_3&=\prod_{i=1}^{x-k-1}\prod_{j=2}^i\frac{2k+i+j-1}{i+j-1}
\label{edcd}
\\
P_4&=\prod_{j=1}^{k}\frac{(x-k-p+2j-1)_{2x+2k-4j+3}}{(2j-1)_{2x+2k-4j+3}}
\label{edce}
\\
P_5&=\prod_{j=1}^{x-k}\frac{(x+k-p+j)_{2x-2k-2j+1}}{(2k+j)_{2x-2k-2j+1}}.
\label{edcf}
\end{align}
We have
\begin{equation}
P_1=\frac{(x-1)!}{k!(x-k-1)!}=\frac{\Gamma(x)}{k!\Gamma(x-k)}~\sim\frac{1}{k!}x^k,\ \ \ x\to\infty,
\label{edcg}
\end{equation}
where we used the classical formula (see e.g. \cite{Olver}, (5.02)/p.\! 119)

\begin{equation}
\frac{\Gamma(z+a)}{\Gamma(z+b)}\sim z^{a-b},\ \ \ z\to\infty.
\label{edch}
\end{equation}
Expressing the Pochhammer symbols as ratios of Gamma functions using
\begin{equation}
(x)_k=\frac{\Gamma(x+k)}{\Gamma(x)},
\label{edci}
\end{equation}
the product $P_2$ can be written as
\begin{equation}
P_2=\prod_{i=0}^{p-1}
\frac{\Gamma(3x+p-k-2i-1)}{\Gamma(2x+p-2i)}
\frac{\Gamma(x+k+p-2i)}{\Gamma(2x+p-2i-1)}.
\label{edcj}
\end{equation}
Using Stirling's approximation (see e.g. \cite{Olver}, (8.16)/p.\! 88)
\begin{equation}
\Gamma(x)\sim e^{-x}x^x\left(\frac{2\pi}{x}\right)^{\frac12},\ \ \ x\to\infty
\label{edck}
\end{equation}
it follows that
\begin{equation}
P_2\sim\frac{3^{3xp-pk-\frac{p}{2}}}{2^{4xp}},\ \ \ x\to \infty.
\label{edcl}
\end{equation}
To find the asymptotics of $P_3$, write
\begin{align}
P_3
&=
\prod_{i=1}^{x-k-1}\frac{(2k+i+1)_{i-1}}{(i+1)_{i-1}}
=
\prod_{i=1}^{x-k-1}\frac{\Gamma(2k+2i)}{\Gamma(2k+i+1)}\frac{\Gamma(i+1)}{\Gamma(2i)}
\nonumber
\\
&=
\Gamma(3)\Gamma(5)\cdots\Gamma(2k+1)
\frac
{\Gamma(2x-2)\Gamma(2x-4)\cdots\Gamma(2x-2k)}
{\Gamma(x+k)\Gamma(x+k-1)\cdots\Gamma(x-k+1)}.
\label{edcm}
\end{align}
Using the identity (see e.g. \cite{Olver}, (1.08)/p.\! 35)
\begin{equation}
\Gamma(2z)=\frac{2^{2z-1}}{\pi^\frac12}\Gamma(x)\Gamma\left(x+\frac12\right)
\label{edcn}
\end{equation}
and \eqref{edch}, we obtain from \eqref{edcm} that
\begin{align}
P_3
&=\Gamma(3)\Gamma(5)\cdots\Gamma(2k+1)
\frac{2^{(2x-3)+(2x-5)+\cdots+(2x-2k-1)}}{\pi^\frac{k}{2}}
\nonumber
\\
&\ \ \ \ \ \ \ \ \ \ \ \ \ \ \ \ 
\times
\frac{\Gamma(x-1)\Gamma\left(x-\frac12\right)\Gamma(x-2)\Gamma\left(x-\frac32\right)\cdots
\Gamma(x-k)\Gamma\left(x-k+\frac12\right)}
{\Gamma(x+k)\Gamma(x+k-1)\cdots\Gamma(x-k+1)}
\nonumber
\\
&\ \ \ \ \ \ \ \ \ \ \ \ \ \ \ \ 
\sim \Gamma(3)\Gamma(5)\cdots\Gamma(2k+1)
\frac{2^{2kx-k(k+2)}}{\pi^\frac{k}{2}}x^{-k\left(k+\frac32\right)},\ \ \ x\to\infty.
\label{edco}
\end{align}
The product $P_4$ can be handled in a similar fashion. We obtain
\begin{align}
P_4
&=
\prod_{j=1}^k
\frac{\Gamma(3x+k-p-2j+2)}{\Gamma(x-k-p+2j-1)}
\frac{\Gamma(2j-1)}{\Gamma(2x+2k-2j+2)}
\nonumber
\\
&\ \ \ \ \ \ \ \ \ \ \ \ \ \ \ \ 
\sim\frac{\Gamma(1)\Gamma(3)\cdots\Gamma(2k-1)}{(3\pi)^\frac{k}{2}}
\frac{3^{3kx+k(1-p)}}{2^{2kx+k(k+1)}}x^{-k\left(k-\frac12\right)},\ \ \ x\to\infty.
\label{edcp}
\end{align}
Using the notation
\begin{equation}
[\Gamma(x)]_k:=\Gamma(x)\Gamma(x+1)\cdots\Gamma(x+k-1),
\label{edcq}
\end{equation}
one can write $P_5$ as
\begin{equation}
P_5=
\frac
{
\dfrac{[\Gamma(2x-p+1)]_p}{[\Gamma(3x-k-p+1)]_{k+p}}[\Gamma(2x+1)]_x
\dfrac{[\Gamma(x+1)]_k}{[\Gamma(1)]_{2k}}[\Gamma(1)]_x
}
{
\dfrac{[\Gamma(x+k-p+1)]_{p-k}}{[\Gamma(2x-p+1)]_{p}}[\Gamma(x+1)]_x
\dfrac{1}{[\Gamma(x+1)]_{k}}[\Gamma(x+1)]_x
}.
\label{edcr}
\end{equation}
Using the recurrence relation $\Gamma(z+1)=z\Gamma(z)$, we can then express the asymptotics of $P_5$ in terms of Barnes' $G$-function\footnote{ It suffices for us to use that Barnes' $G$-function satisfies the recurrence relation $G(z+1)=\Gamma(z)G(z)$.} as
\begin{equation}
P_5
\sim
\frac{2^{4px+k-p^2}}{3^{(k+p)(6x-k-p)/2}}\frac{\pi^k}{\Gamma(1)\Gamma(2)\cdots\Gamma(2k)}
\frac{G(3x)G(x)^3}{G(2x)^3}x^{2k^2},\ \ \ k\to\infty.
\label{edcs}
\end{equation}
Substituting the above asymptotics relations into \eqref{edca} (and using that the asymptotics of $P_6$ is just the $k=p=0$ specialization of \eqref{edcs}), we obtain after simplifications the statement of the theorem. \end{proof}

We only determine the bulk correlation $\omega(k,p)$ in the case $p=0$. What allows us to do this is exact product formulas we found in earlier work \cite{symffb} for the enumeration of lozenge tilings of what we call axial shamrock regions. For $p>0$ $\M(H_{x,x}(k,p)$ does not seem to be given by a simple product formula, and there are no other currently known manageable expressions for it that would allow one to compute the bulk correlation \eqref{edb}.

\begin{theo}
\label{tdb}
For non-negative integers $k$ we have
\begin{align}
&
\omega(k,0)=\frac{1}{\pi^k} 
\frac{\Gamma(2k+1)\Gamma\left(k+\frac12\right)}
{\Gamma(k+1)\Gamma\left(2k+\frac12\right)}
\left[
\prod_{i=1}^k
\frac
{\Gamma(i)\Gamma(2i-1)}
{\Gamma\left(k+i-\frac12\right)}
\right]^2
\frac
{3^{2k^2}}
{2^{6k^2+2k}}.
\label{edd}
\end{align}
\end{theo}

\begin{proof} 
The starting point for the proof is the explicit product formula for the numerator in the $p=0$ specialization of \eqref{edb} provided by Corollary 2.3 and Theorems 2.1 and 2.2 of our earlier work \cite{symffb}. The asymptotics of this product formula can be analyzed in a manner similar to the one presented in the proof of Theorem \ref{tda}. One arrives at the asymptotic formula \eqref{edd}. \end{proof}

\begin{cor}
\label{tdc}
The asymptotics of the bulk correlation $\omega(k,0)$ is
\begin{equation}
\omega(k,0)\sim
\frac{e^{\frac14}}{A^3 2^\frac16 k^\frac14}
\frac{3^{2k^2}}{2^{8k^2}},\ \ \ k\to\infty,
\label{ede}
\end{equation}
where $A=1.2824271291...$ is the Glaisher-Kinkelin constant.\footnote{The Glaisher-Kinkelin constant (see \cite{Glaish}) is the value $A$ for which
$\lim_{n\to\infty}
\dfrac
 {0!\,1!\,\cdots\,(n-1)!}
 {n^{\frac{n^2}{2}-\frac{1}{12}}\,(2\pi)^{\frac{n}{2}}\,e^{-\frac{3n^2}{4}}}
=
\dfrac
 {e^{\frac{1}{12}}}
 {A}
$.}
\end{cor}

\begin{proof} 
This follows using \eqref{edch} and the asymptotic relation that defines  the Gaisher-Kinkelin constant $A$ (see footnote 3).
\end{proof}

It is clear from Theorem \ref{tda} and Corollary \ref{tdc} that corner correlation $\omega_c(k,0)$ and the fourth root of the bulk correlation $\omega(k,0)$ do not have (up to a multiplicative constant) the same asymptotics as $\to\infty$. The fact that this agreement in the asymptotics, which holds in the set-up of \cite{rangle} mentioned above, fails here, is not so surprising: In \cite{rangle}, as the arguments of the corner correlation approach infinity, the defects (a triangular hole of side 2 in that case) are removed infinitely far from the boundary; by contrast, as $k\to\infty$, the dent in the corner whose effect is recorded by $\omega_c(k)$ still starts at the corner of the boundary.

What is remarkable is that $\omega_c(k,0)$ and $\omega(k,0)^{\frac14}$ do have the same log-asymptotics.\footnote{ We say that $f(n)$ and $g(n)$ have the same log-asymptotics if $\ln f(n) \sim \ln g(n)$, $n\to\infty$.} Given the parallels between the correlation of gaps in dimer systems and 2D electrostatics we found in previous work (see \cite{sc}\cite{ec}\cite{ef}\cite{ov}), and in particular that the electrostatic potential corresponds to the logarithm of the correlation, we can view the next result as stating that the method of images from electrostatics still works in this new circumstance.

\begin{cor}
\label{tdd}
The corner correlation $\omega_c(k,0)$ and the fourth root of the bulk correlation $\omega(k,0)$ have the same log-asymptotics:
\begin{equation}
\ln \omega_c(k,0) \sim \ln \omega(k,0)^{\frac14} \sim k^2\ln\frac{\sqrt{3}}{4},\ \ \ k\to\infty.
\end{equation}

\end{cor}

\begin{proof} This follows directly from Theorem \ref{tda} and Corollary \ref{tdc}. \end{proof}

\section{Concluding remarks}

In this paper we have enumerated the lozenge tilings of a hexagon with a shamrock removed from its center that are in a fourth symmetry class, that extending the class of symmetric and self complementary plane partitions. The remaining two cases (which correspond to self complementary, resp. symmetric plane partitions) will be presented in subsequent work.

One interesting feature of our proof is that for it to work, we needed to generalize the regions under consideration, and we ended up proving a simple product formula for these more general regions. There are natural counterpart regions generalizing the base case, but they are not round.

We have also analyzed the asymptotics of the corner correlation of a macroscopic dent in a 90 degree wedge with mixed boundary conditions, and found that it has the same log-asymptotics as the fourth root of the bulk correlation of the region obtained by reflecting the dent in the two sides of the wedge. This represents an analog of the method of images from electrostatics which turns out to hold in this circumstance as well (in the presence of a macroscopic dent touching the boundary). The analogy to electrostatics may be deeper than previously thought.

\end{document}